\documentclass[11pt]{article}
\usepackage{hyperref}
\usepackage{relsize}
\usepackage{graphicx}
\usepackage{amsmath, amsthm, amssymb, amscd, mathrsfs, eucal, epsfig,color}
\usepackage[english]{babel}

\usepackage{pinlabel}
\usepackage{pgf}
\usepackage{setspace}

\usepackage{geometry}
\newtheorem{theorem}{Theorem}[section]
\newtheorem{corollary}[theorem]{Corollary}
\newtheorem{lemma}[theorem]{Lemma}
\newtheorem{prop}[theorem]{Proposition}
\newtheorem*{prop*}{Proposition}

\newtheorem*{theorem*}{Theorem}
\newtheorem*{corollary*}{Corollary}
\newtheorem*{lemma*}{Lemma}

\theoremstyle{definition}
\newtheorem{definition}{Definition}[section]
\theoremstyle{remark}
\newtheorem{rem}{Remark}[section]

\newcommand{\Sph}{\mathbb{S}}

\newcommand{\C}{\mathcal{C}}
\newcommand{\U}{\mathcal{U}}
\newcommand{\Or}{\mathcal{O}}

\newcommand{\Pc}{\mathcal{P}}
\newcommand{\F}{\mathcal{F}}
\newcommand{\Qc}{\mathcal{Q}}
\newcommand{\R}{\mathbb{R}}
\newcommand{\N}{\mathbb{N}}
\newcommand{\Z}{\mathbb{Z}}

\newcommand{\al}{(\alpha_i^\pm)_i}
\newcommand{\bet}{(\beta_i^\pm)_i}

\newlength{\larg}
\begin{document}
\title{Conjugacy invariants for Brouwer mapping classes}
\author{Juliette Bavard\footnote{Supported by grants from R\'{e}gion Ile-de-France.}}
\date{\today}

\maketitle
\begin{abstract} 
We give new tools for homotopy Brouwer theory. In particular, we describe a canonical reducing set (the set of \emph{walls}) which splits the plane into \emph{maximal translation areas} and \emph{irreducible areas}. We then focus on Brouwer mapping classes relatively to four orbits and describe them explicitly by adding to Handel's diagram and to the set of walls a \emph{tangle}, which is essentially an isotopy class of simple closed curves in the cylinder minus two points.
\end{abstract}

\setcounter{tocdepth}{1}
\tableofcontents

\section*{Introduction}
\subsection*{Homotopy Brouwer theory}
Homotopy Brouwer theory was introduced by Michael Handel in \cite{Handel} to prove his famous fixed point theorem for planar homeomorphism, which has many applications to surface homeomorphisms (for examples, see the introduction of \cite{LeCalvez}). This theory was mainly used by John Franks and Michael Handel, e.g. to study Hamiltonian surface diffeomorphisms in \cite{Franks-Handel-Periodic-points-of-Hamiltonian-surface-diffeomorphisms} and prove the Zimmer conjecture for area preserving diffeomorphisms of surfaces in \cite{Franks-Handel-Area-preserving-group-actions-on-surfaces}.

Homotopy Brouwer theory can be seen as the study of the elements of the mapping class group of the plane minus $\Z$ which are classes of Brouwer homeomorphisms relatively to finitely many orbits. More precisely, consider a Brouwer homeomorphism $h$, i.e. a fixed-point-free homeomorphism of the plane preserving the orientation. Choose finitely many disjoint orbits of this homeomorphism and denote by $\Or$ their union. Classical Brouwer theory tells us that each orbit of a Brouwer homeomorphism is properly embedded in the plane, i.e. intersects every compact set of the plane in only finitely many points (see for example \cite{Guillou}). Hence $\Or$ is homeomorphic to $\Z$ in the plane. Denote by $MCG(\R^2,\Or)$ the mapping class group of the plane relatively to $\Or$, i.e. the quotient of the group of orientation preserving homeomorphisms of the plane which globally fix $\Or$ by its connected component of the identity for the compact-open topology. Now we can look at the class of $h$ in $MCG(\R^2,\Or)$: since $h$ is a Brouwer homeomorphism, this class is said to be a \emph{Brouwer mapping class}. We denote it by $[h;\Or]$. Two Brouwer mapping classes $[h;\Or]$ and $[h';\Or']$ are said to be conjugate if there exists an orientation-preserving homeomorphism $\phi$ of the plane such that $\phi(\Or)=\Or'$ and $[\phi h \phi^{-1};\phi(\Or)]$ is equal to $[h';\Or']$ in $MCG(\R^2;\Or')$. One aim of homotopy Brouwer theory is to describe (up to conjugacy) every Brouwer mapping class relatively to finitely many given orbits.\\

\subsection*{Brouwer mapping classes relatively to one, two, and three orbits}

In \cite{Handel}, Michael Handel gives a complete description of Brouwer mapping classes relatively to one and two orbits. He shows that relatively to one orbit, there exists only one Brouwer mapping class up to conjugacy: the class of the translation relatively to one of its orbits. Relatively to two orbits, he proves that they are exactly three Brouwer mapping classes (up to conjugacy): the class of the translation, the class of the time one map $R$ of the Reeb flow and the class of $R^{-1}$.

In \cite{FLR2013}, Fr\' ed\' eric Le Roux gives a complete description of Brouwer mapping classes relatively to $3$ orbits and uses this description to define an index for Brouwer homeomorphisms. In particular, he shows that there are only finitely many Brouwer mapping classes relatively to $3$ orbits, and that each of them contains the time one map of a flow (see \cite{FLR2013} for more details and the complete description of this classes). 

The situation changes if we look at the Brouwer mapping classes relatively to more than $3$ orbits: indeed, if $r\geq 4$, there are infinitely many Brouwer mapping classes relatively to $r$ orbits, and only finitely many of them contain the time one map of a flow. One aim of this paper is to give a complete description of Brouwer mapping classes relatively to $4$ orbits.\\

\subsection*{Walls} In \cite{Handel}, Michael Handel defines reducing lines of a Brouwer mapping class $[h;\Or]$: such a line is homotopic to its image by $h$ and splits the set of orbits into two smaller sets. He proves that every Brouwer mapping class relatively to more than one orbits has at least one reducing line (theorem $2.7$ of \cite{Handel}). 

We propose to call the isotopy class of a reducing line $\Delta$ a \emph{wall} if every other reducing line is homotopically disjoint from $\Delta$. The set of walls is clearly a conjugacy invariant for Brouwer mapping classes. We prove the following result:

\begin{theorem*}\emph{\textbf{\ref{theorem: walls}}}
Let $[h;\Or]$ be a Brouwer mapping class. Let $\mathcal{W}$ be a family of pairwise disjoint reducing lines containing exactly one representative of each wall for $[h;\Or]$. If $Z$ is a connected component of $\R^2 - \mathcal{W}$, then exactly one of the followings holds:
\begin{itemize}
\item $Z$ is an irreducible area;
\item $Z$ is a maximal translation area;
\item $Z$ does not intersect $\Or$.
\end{itemize}
\end{theorem*}

Precise definitions of irreducible and maximal translation areas will be given in section \ref{section:statements}. A translation area $Z$ is in particular an area which is invariant under $h$ and on which $h$ has a very simple dynamics: indeed, up to conjugacy, $h$ is conjugated to a homeomorphism whose restriction to $Z$ is a translation. Contrariwise, an irreducible area has a more complex dynamics and cannot be reduced into more simple areas: it does not contain any reducing line. It will follow that if the complement of the walls of a Brouwer mapping class does not contain any irreducible area, then we will be able to understand easily the Brouwer mapping class, which will indeed be a time one map of a flow (see section \ref{section:statements}). Moreover, we will prove that if the complement of the walls of a Brouwer mapping class has an irreducible area, then it has at least also two maximal translation areas.

\subsection*{Diagrams}
Following essentially \cite{Handel} and \cite{FLR2013}, we can associate to every conjugacy class of Brouwer mapping class a unique diagram, for which a precise definition will be given in section \ref{section:statements}. This diagram is a disk with $r$ arrows, where $r$ is the number of orbits that we consider: each arrow represents an orbit. The cyclic order of the endpoints of the arrows is determined by the existence of a \emph{nice family of homotopy translation arcs} (see section \ref{section:statements}). A diagram of Brouwer mapping class is said to be \emph{determinant} if there exists only one conjugacy class of Brouwer mapping class associated to it. Every diagram for Brouwer mapping class relatively to one, two or three orbits is determinant. For $4$ orbits or more, there exist diagrams which are not determinant.

\begin{figure}[h]
\centering
\includegraphics[scale=1]{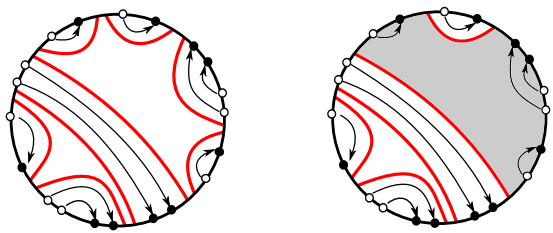}
\caption{Examples of diagrams with walls: the one on the left is determinant, and the one on the right is non determinant}
\label{figu:diag-with-walls}
\end{figure}

We can add the set of walls on this diagram: we obtain a \emph{diagram with walls}, which is again a conjugacy invariant for Brouwer mapping classes (see figure \ref{figu:diag-with-walls} for examples). This invariant is more precise than the diagram without walls, but it is still not total for Brouwer mapping classes relatively to more than $3$ orbits. Again we can define the notion of \emph{deteminant} diagram with walls (which corresponds to only one conjugacy class). We give an elementary combinatoric condition to identify the determinant diagrams among diagrams with walls without crossing arrows:
\begin{prop*}\emph{\textbf{\ref{coro:determinant diagrams with walls}}} A diagram with walls without crossing arrows is determinant if and only if the arrows of every family of arrows included in the same connected component of the complement of the walls are backward adjacent and forward adjacent.  \end{prop*}

\subsection*{A total conjugacy invariant for Brouwer mapping class relatively to $4$ orbits}

For Brouwer mapping classes relatively to $4$ orbits, we add a new invariant to the non determinant diagrams with walls: the \emph{tangle}. This invariant is an isotopy class of curves on the cylinder with two marked points (up to horizontal twists). See figure \ref{figu:intro} for an example. Using in particular the set of walls and the description of determinant diagrams with $4$ orbits, we get a total conjugacy invariant:

\begin{theorem*}\emph{\textbf{\ref{theo:total invariant}}}
Two Brouwer mapping classes relatively to $4$ orbits are conjugated if and only if they have the same couple (Diagram with walls, Tangle).
\end{theorem*}

\begin{figure}[h]
\centering
\includegraphics[scale=1]{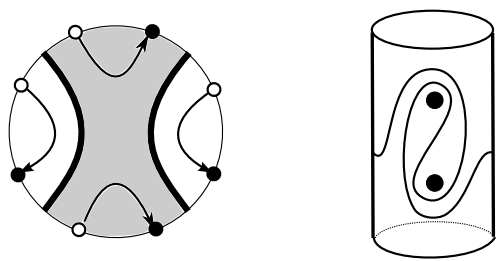}
\caption{Example of a couple (Diagram with walls, Representative of the tangle)}
\label{figu:intro}
\end{figure}

We finally get a complete description of Brouwer mapping classes relatively to four orbits.

In the first section, we recall useful tools for homotopy Brouwer theory from \cite{Handel} and \cite{FLR2013}. A precise description of the results will be given in section \ref{section:statements}. The remaining of the text is devoted to proofs.

\subsection*{Acknowledgment}

I would like to thank my advisor Fr\'{e}d\'{e}ric Le Roux for his many advices and explanations and for his careful readings of the different versions of this text.

\section{First tools of homotopy Brouwer theory}\label{section: First tools}

We recall the following definitions and properties (see \cite{Handel}, \cite{Frederic-An-introduction-to} and \cite{FLR2013}).
Let $[h;\Or]$ be a Brouwer mapping class, i.e. the isotopy class of a Brouwer homeomorphism $h$ relatively to a finite set of orbits $\Or$. Denote by $r$ the number of orbits of $\Or$. We choose a complete hyperbolic metric of the first kind on $\R^2 - \Or$. Even if not explicitly specified, we will always consider complete hyperbolic metrics \emph{of the first kind} on surfaces, i.e. such that the surface is isomorphic to $\mathbb{H}^2/\Gamma$ where $\Gamma$ is of the first kind (see Matsumoto \cite{Matsumoto} for details). 

\subsection{Examples: flows and product with a free half twist}
\noindent \textbf{Flows.} For abbreviation, we say that a homeomorphism $f$ is a flow if it is the time one map of a flow. If a Brouwer homeomorphism $f$ is isotopic to a flow relatively to $\Or$, then we say that $[f;\Or]$ is a \emph{flow class}. \\

\noindent \textbf{Example A.} The first example is the flow class of figure \ref{figu:Reeb1}. In this example, we choose $5$ streamlines of a flow $f$ and get a Brouwer mapping class relative to $5$ orbits: $\Or_1, \Or_2, \Or_3, \Or_4$ and $\Or_5$. We denote by $\Or$ their union.

\begin{figure}[h] 
\labellist
\small\hair 2pt	
\pinlabel $\Or_1$ at 390 10
\pinlabel $\Or_2$ at 390 114
\pinlabel $\Or_5$ at 390 67
\pinlabel $\Or_3$ at 335 82
\pinlabel $\Or_4$ at 445 44
\endlabellist
\centering
\vspace{0.2cm}
\includegraphics[scale=0.75]{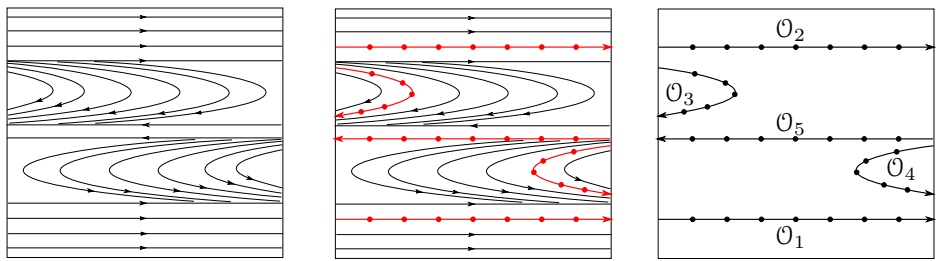}
\caption{Example of a Brouwer mapping class $[f;\Or]$ relatively to $5$ orbits.}
\label{figu:Reeb1}
\end{figure}

\noindent \textbf{Free half twist.} We call \emph{half twist} any homeomorphism which is:
\begin{itemize}
\item supported in a topological disk $D$ of $\R^2$ which contains exactly $2$ points of $\Or$, denoted by $x$ and $y$;
\item isotopic to a homeomorphism supported in $D$ which is a rotation of a half turn on a disk included in $D$, which exchanges $x$ and $y$.
\end{itemize} 

If $h$ is a Brouwer homeomorphism, we call \emph{$h$-free half twist} every half twist $\mu$ supported in a $h$-free disk, i.e. in a disk $D$ such that $h^n(D)\cap D=\emptyset$ for every non zero $n\in \Z$. Note that $\mu h$ is a Brouwer homeomorphism.\\

\begin{figure}[h]
\labellist
\small\hair 2pt	
\pinlabel $\mu$ at 68 50
\pinlabel $x_1$ at 85 12
\pinlabel $x_2$ at 85 112
\pinlabel $x_3$ at 19 84
\pinlabel $x_4$ at 120 59
\endlabellist
\centering
\vspace{0.2cm}
\includegraphics[scale=0.75]{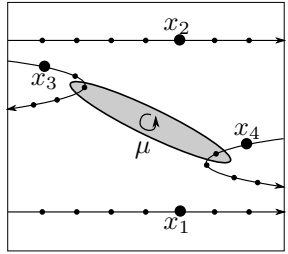}
\caption{Example $B$: the Brouwer mapping class $[g;\Or']$.}
\label{figu:ReebB1}
\end{figure}

\noindent \textbf{Example B.} Our second example is the product of $f$ with the free half twist $\mu$ of \mbox{figure \ref{figu:ReebB1}}, which exchanges the two points of the disk which are in $\Or$. We denote by $g$ the product $\mu f$. For $i=1,2,3,4$, we choose $x_i$ on $\Or_i$ as in figure \ref{figu:ReebB1} and denote by $\Or'_i$ the $g$-orbit of $x_i$, i.e. $\{g^n(x_i)\}_{n\in \Z}$. In particular, we have $\Or_1=\Or'_1$ and $\Or_2=\Or'_2$. We denote by $\Or'$ the union of the $\Or'_i$. Note that $\Or'$ coincide with $\Or_1 \cup \Or_2 \cup \Or_3 \cup \Or_4$. We consider the Brouwer mapping class $[g;\Or']$.

\subsection{Arcs and topological lines in $\R^2 - \Or$.} We call \emph{arc} an embedding $\alpha$ of $]0,1[$ in $\R^2 - \Or$ such that it can be continuously extending to $0$ and $1$ with $\alpha(0),\alpha(1) \in \Or$. By abuse of notations, we will call again arc and denote by $\alpha$ the image of the embedding $\alpha(]0,1[)$. The extensions $\alpha(0)$ and $\alpha(1)$ are said to be the \emph{endpoints} of $\alpha$. A \emph{topological line} is a proper embedding $\alpha$ of $\R$ in $\R^2$, i.e. an embedding $\alpha$ such that for every compact $K$ of the plane, there exists $t_0 \in \R$ such that if $|t|>|t_0|$, then $\alpha(t) \notin K$. Again, by abuse of notations, we call (topological) line and denote by $\alpha$ the image of $\R$ by $\alpha$ in $\R^2 - \Or$.

\subsection{Isotopy classes of arcs and lines.} We say that two arcs (respectively two lines) $\alpha$ and $\beta$ are isotopic relatively to $\Or$ if there exists a continuous and proper application $H : ]0,1[ \times [0,1] \rightarrow \R^2 - \Or$ such that:
\begin{itemize}
\item $H(\cdot,0)=\alpha(\cdot)$ and $H(\cdot,1)=\beta(\cdot)$;
\item If $\alpha$ and $\beta$ are arcs, $H$ can be continuously extending to $[0,1]\times [0,1]$ in such a way that the endpoints coincide, i.e. for every $t \in ]0,1[$, we have: $\alpha(0)=H(0,t)=\beta(0)$ and $\alpha(1)=H(1,t)=\beta(1)$. 
\end{itemize}

If $\alpha$ is an arc or a line of $\R^2 - \Or$, we denote by $\alpha_\#$ the geodesic representative in the isotopy class of $\alpha$ relatively to $\Or$. It is known that this geodesic representative is unique and that if $\alpha$ and $\beta$ are two arcs or lines, then $\alpha_\#$ and $\beta_\#$ are in minimal position. In particular, if $\alpha$ and $\beta$ are homotopically disjoint (i.e. they have disjoint representatives in their isotopy classes), then $\alpha_\#$ and $\beta_\#$ are disjoint.

\subsection{Straightening principle.}
We will need the following lemma, which is the lemma $3.5$ of \cite{Handel} (see also lemma $1.4$, corollary $1.5$ and lemma $3.2$ of \cite{FLR2013}):
\begin{lemma}[Straightening principle] Let $\F_1$ and $\F_2$ be two locally finite families of lines and arcs of $\R^2 -\Or$ such that:
\begin{itemize} \label{lemma: straightening principle}
\item The elements of $\F_1$ (respectively $\F_2$) are mutually homotopically disjoint;
\item If $\alpha \in \F_1$ and $\beta \in \F_2$, then $\alpha$ and $\beta$ are in minimal position.
\end{itemize}
Then the following statements hold.
\begin{enumerate}
\item There exists a homeomorphism $u$ isotopic to $Id$ relatively to $\Or$ such that for every element $\gamma$ in $\F_1 \cup \F_2$, $u(\gamma)=\gamma_\#$, where $\gamma_\#$ is the geodesic representative of the isotopy class of $\gamma$;
\item If $h$ is an orientation preserving homeomorphism of the plane such that $h(\Or)=\Or$, then there exists $h'\in [h;\Or]$ such that for every $\alpha$ in $\F_1 \cup \F_2$, we have $h'(\alpha_\#)=h(\alpha)_\#$.
\end{enumerate}
\end{lemma}

We get $2$ by applying $1$ to $h(\F_1 \cup \F_2)$.

\subsection{Homotopy translation arc.}
A \emph{homotopy translation arc} for $[h;\Or]$ is an arc $\alpha$ such that:
\begin{itemize}
\item There exists $x \in \Or$ such that $\alpha(0)=x$ and $\alpha(1)=h(x)$;
\item For every $n \in \Z$, $h^n(\alpha)$ is homotopically disjoint from $\alpha$.
\end{itemize}
In particular, every translation arc for $h$ with endpoints in $\Or$ is a homotopy translation arc for $[h;\Or]$. In general, there exist homeomorphisms $h$ and arcs which are not homotopic to translation arc for $h$, but which are homotopy translation arcs for $[h;\Or]$.\\

\begin{figure}[h]
\labellist
\small\hair 2pt	
\pinlabel $\gamma$ at 119 44
\pinlabel $\alpha$ at 65 114
\pinlabel $\beta$ at 80 80
\pinlabel $\Or_1$ at 132 12
\pinlabel $\Or_2$ at 132 112
\pinlabel $\Or_3$ at 8 88
\pinlabel $\Or_4$ at 132 49
\pinlabel $\Or_5$ at 132 66
\endlabellist
\centering
\includegraphics[scale=1]{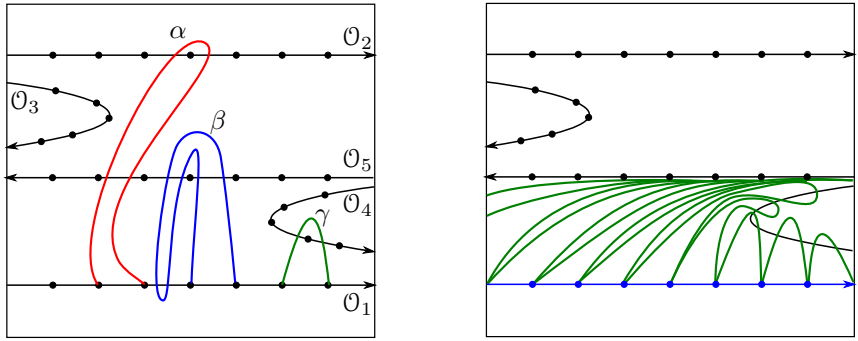}
\caption{Examples of homotopy translation arcs and homotopy streamlines.}
\label{figu:Reeb2}
\end{figure}

\noindent \textbf{Example A:} Figure \ref{figu:Reeb2} (left) shows different homotopy translation arcs for example $A$: the arcs $\beta$ and $\gamma$ are homotopy translation arcs for $[f;\Or']$. The arc $\alpha$ is not a homotopy translation arc for $[f;\Or]$. Note that however if we forget the orbit $\Or_5$, $\alpha$ is a homotopy translation arc for $[h;\Or_1 \cup \Or_2 \cup \Or_3 \cup \Or_4]$.

\subsection{Half homotopy streamlines.}
If $\alpha$ is a homotopy translation arc for $[h;\Or]$, we define the \emph{homotopy streamline}:
$$T(\alpha,h,\Or):=\bigcup_{n\in \Z} (h^n(\alpha([0,1[)))_\#.$$
Since $\alpha$ is a homotopy translation arc, the geodesic iterates are mutually disjoint, hence $T(\alpha,h,\Or)$ is an embedding of $\R$, which can eventually be non proper.

We also define the \emph{backward (respectively forward) homotopy streamline} $T^-(\alpha,h,\Or)$ (respectively $T^+(\alpha,h,\Or)$) by:
$$T^-(\alpha,h,\Or):=\bigcup_{n\leq 0} (h^n(\alpha([0,1[)))_\#.$$
$$T^+(\alpha,h,\Or):=\bigcup_{n\geq 0} (h^n(\alpha([0,1[)))_\#.$$

\noindent \textbf{Example A:} The streamline $T(\beta,f,\Or)$ of example $A$ is proper and coincides with the horizontal streamline which contains $\Or_1$. The streamline $T(\gamma,h,\Or)$ is drawn on figure \ref{figu:Reeb2}. It is not proper, but $T^+(\gamma,h,\Or)$ is proper.

\subsection{Backward proper and forward proper arcs.}
Let $\alpha$ be a homotopy translation arc. If the backward homotopy streamline $T^-(\alpha,h,\Or)$ is the image of $\R^+$ under a proper embedding, i.e. if for every compact $K$ of the plane there exists $n_0 \leq 0$ such that for every $n \leq n_0$, $(h^n(\alpha))_\#$ does not intersect $K$, then $\alpha$ is said to be a \emph{backward proper arc} for $[h;\Or]$. Similarly, if the forward homotopy streamline $T^+(\alpha,h,\Or)$ is proper, then $\alpha$ is said to be a \emph{forward proper arc} for $[h;\Or]$.\\

\begin{figure}[h]
\labellist
\small\hair 2pt	
\pinlabel $\Or'_1$ at 132 12
\endlabellist
\centering
\includegraphics[scale=1]{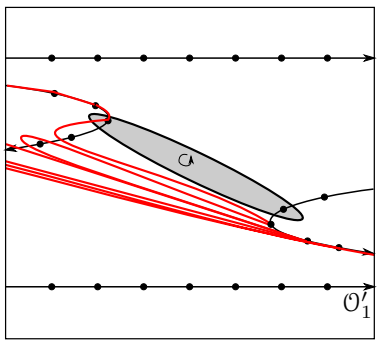}
\caption{Example of a homotopy streamline for $[g;\Or']$.}
\label{figu:ReebB2}
\end{figure}

\noindent \textbf{Example A:} In the example of figure \ref{figu:Reeb2}, $\beta$ is a backward proper and forward proper arc, and $\gamma$ is a forward proper arc but it is not a backward proper arc. Note that if $h$ is a flow, then every homotopy translation arc which lies on a flow streamline is backward proper and forward proper.\\

\noindent \textbf{Example B:} In figure \ref{figu:ReebB2}, we draw iterates of an arc lying on a streamline of $f$ which intersects the support of the free half twist $\mu$. We see that this arc is backward proper but not forward proper (the iterates are ``stuck'' by the orbit $\Or'_1$).\\

\noindent \textbf{Translation.} If a Brouwer homeomorphism $f$ is conjugate to a translation relatively to $\Or$, then we say that $[f;\Or]$ is a \emph{translation class}. In this case, every homotopy translation arc for $[f;\Or]$ is backward proper and forward proper.

\subsection{Nice family $(\alpha_i^\pm)_{1\leq i \leq r}$.}
A \emph{nice family} $(\alpha_i^\pm)_{1\leq i \leq r}$ associated to $[h;\Or]$ is a family of homotopy translation arcs for $[h;\Or]$ such that:
\begin{itemize}
\item For every $1 \leq i \leq r$:
\begin{itemize}
\item $\alpha_i^-$ is a backward proper arc;
\item $\alpha_i^+$ is a forward proper arc;
\item $\alpha_i^-$ and $\alpha_i^+$ have the same endpoints, lying in the orbit $\Or_i$;
\end{itemize}
\item The backward proper half streamlines $T^-(\alpha_i^-,h,\Or)$'s are mutually disjoint;
\item The forward proper half streamlines $T^+(\alpha_i^+,h,\Or)$'s are mutually disjoint.
\end{itemize}

Note that if $\al$ is a nice family for a Brouwer mapping class $[h;\Or]$, then the previous proper half streamlines $T^\pm(\alpha_i^\pm,h,\Or)$'s are mutually disjoint outside a topological disk of the plane.\\

\begin{figure}[h]
\labellist
\small\hair 2pt	
\pinlabel $\alpha_3^+$ at 250 90
\pinlabel $\alpha_4^-$ at 225 50
\endlabellist
\centering
\includegraphics[scale=1]{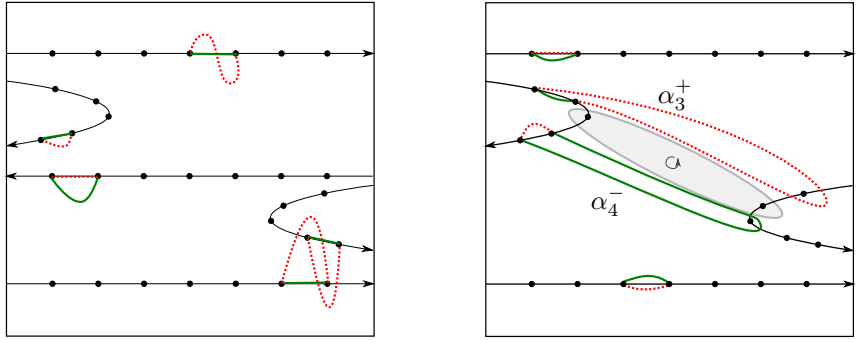}
\caption{Example of nice families for $[f;\Or]$ and $[g;\Or']$.}
\label{figu:Reeb4}
\end{figure}

\noindent \textbf{Examples:} Figure \ref{figu:Reeb4} give an example of a nice family for $[f;\Or]$ with some arcs not homotopic to arcs included in streamlines, and an example of a nice family for $[g;\Or']$. In particular, $\alpha_3^+$ (respectively $\alpha_4^-$) is constructed with an iteration by $g^{-1}$ (respectively $g$) of an arc lying on the $f$-streamline of $\Or_4$ after the support of $\mu$ (respectively of an arc lying on the $f$-streamline of $\Or_3$ before the support of $\mu$).\\

The following theorem of Handel \cite{Handel} allows us to consider a nice family for every Brouwer mapping class in the following sections.
\begin{theorem}[Handel \cite{Handel}] \label{theo: Handel} For every $[h;\Or]$, there exists a nice family associated to $[h;\Or]$.
\end{theorem}

\begin{rem} \label{remark:generalized homotopy streamlines}
For a statement closer to this one, see \cite{FLR2013}, proposition $3.1$. Here we describe a way to deduce our statement from proposition $3.1$ of \cite{FLR2013}, where the statement is given for \emph{generalized homotopy half streamlines}. As seen in figure \ref{figu:generalized_half_streamlines}, in any open neighborhood of a generalized homotopy half streamline there exist disjoint homotopy streamlines whose union contains every points of $\Or$ included in the generalized homotopy half streamline. It follows that the result is still true with the statement given here (i.e. when we replace disjoint generalized homotopy half streamlines by (non generalized) disjoint homotopy half streamlines).
\end{rem}

\begin{figure}[h]
\centering
\vspace{0.2cm}
\includegraphics[scale=1]{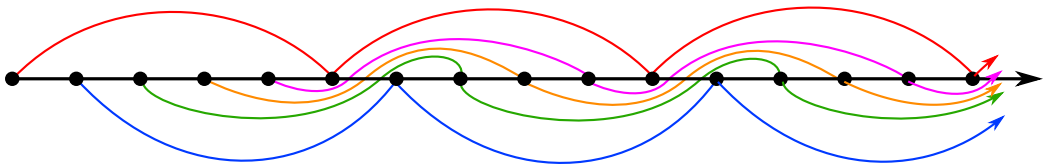}
\caption{Finding disjoint homotopy half streamlines in a neighborhood of a generalized homotopy half streamline.}
\label{figu:generalized_half_streamlines}
\end{figure}

\subsection{Reducing line.}

We say that a line of $\R^2$ \emph{splits} a given set of points $X$ included in $\R^2 -\Delta$ if both connected components of $\R^2 -\Delta$ intersect $X$.

A \emph{reducing line} $\Delta$ for $[h;\Or]$ is a line in $\R^2 - \Or$ such that $h(\Delta)$ is properly isotopic to $\Delta$ relatively to $\Or$ and such that $\Delta$ splits $\Or$. Note that all the elements of a same orbit of $\Or$ are included in the same connected component of $\R^2 -\Delta$. Indeed, according to the straightening principle \ref{lemma: straightening principle}, there exists $h'\in [h;\Or]$ such that $h'(\Delta)=\Delta$. Figure \ref{figu:ReebB4} gives examples.

\begin{figure}[h]
\centering
\includegraphics[scale=1]{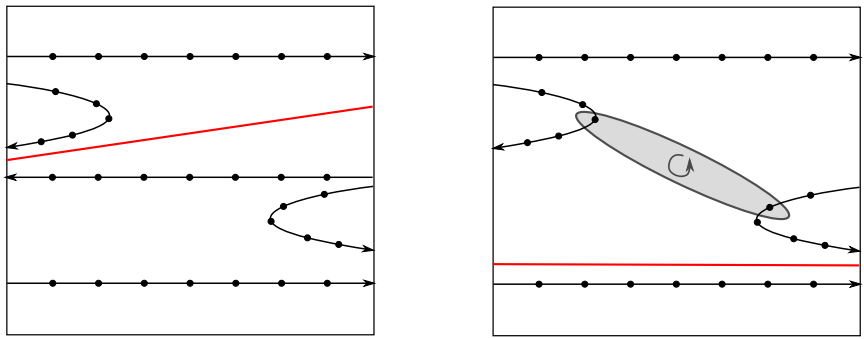}
\caption{Examples of reducing lines for $[f;\Or]$ and $[g;\Or']$.}
\label{figu:ReebB4}
\end{figure}

We will need the following Handel's theorem. See \cite{FLR2013}, proposition $3.3$ for this formulation.

\begin{theorem}[Handel \cite{Handel}]\label{theo:existence of reducing line}
Let $[h;\Or]$ be a Brouwer mapping class relatively to more than one orbits. Let $\al$ be a nice family for $[h;\Or]$. There exists a reducing line for $[h;\Or]$ which is disjoint from every backward proper half streamline $T^-(\alpha_i^-,h,\Or)$.
\end{theorem}

\subsection{Homotopy Brouwer line.} 
\begin{figure}[h]
\labellist
\small\hair 2pt	
\pinlabel $L$ at 80 72
\pinlabel $L$ at 263 32
\endlabellist
\centering
\includegraphics[scale=1]{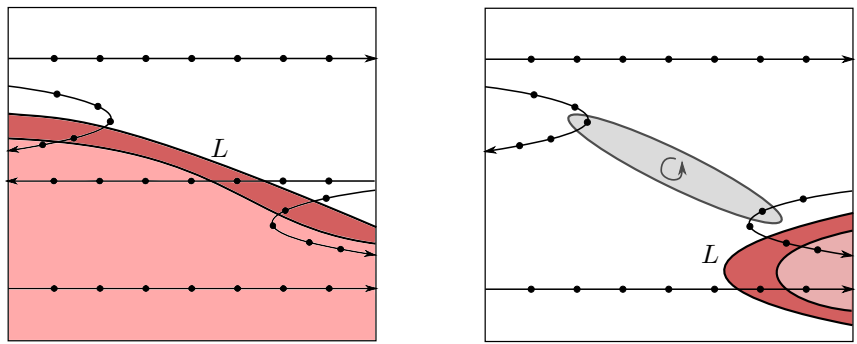}
\caption{Examples of Brouwer lines for $[f;\Or]$ and $[g;\Or']$.}
\label{figu:ReebB5}
\end{figure}

A \emph{homotopy Brouwer line} $L$ for $[h;\Or]$ is a topological line in $\R^2 - \Or$ such that:
\begin{itemize}
\item $L$ is homotopically disjoint from $h(L)$; 
\item $L$ is not isotopic to $h(L)$;
\item If we denote by $V$ the connected components of $\R^2 - L_\#$ containing $h(L)_\#$, then we have $h(V)_\# \subset V$, where $h(V)_\#$ is the connected component of $\R^2 - h(L)_\#$ which does not contain $L_\#$;
\item If $x \in \Or$ is in $V - h(V)_\#$, then $h(x) \notin V - h(V)_\#$.
\end{itemize} 
This definition does not depend on the chosen metric on $\R^2 -\Or$. Figure \ref{figu:ReebB5} gives examples of Brouwer lines for $[f;\Or]$ and $[g;\Or']$.

\section{Description of the results}\label{section:statements} Here we give the main definitions and statements of the paper. Proofs will be given in the following sections.

\subsection{Adjacency areas, diagrams and special nice families (Section \ref{section:Adjacency areas, diagrams and special nice families})}

\subsubsection{Cyclic order of a nice family}
Let $[h;\Or]$ be a Brouwer mapping class and let $(\alpha_i^\pm)_i$ be a nice family for $[h;\Or]$. There is a natural cyclic order on the elements of the nice family $\al$ given by the order of the half homotopy streamlines generated by the $\alpha_i^\pm$ at infinity: if we choose a big enough topological circle which intersects each half streamline only once, with transverse intersections, the order on the half streamlines is given by the order of these intersections (which is independent of the choice of the circle). In the following, we will call this cyclic order the \emph{cyclic order of the nice family}.

\subsubsection{Adjacency} If several forward proper arcs, respectively several backward proper arcs, are consecutive for the cyclic order of the nice family, then they are said to be \emph{adjacent}. A sub-family of the nice family consisting only of consecutive arcs of the same type (all backwards or all forwards) is said to be a \emph{sub-family of adjacency}. If two orbits have forward proper arcs (respectively backward proper arcs) in the same nice family which are adjacent, they are said to be \emph{forward adjacent} (respectively \emph{backward adjacent}). The following proposition is essentially due to Handel \cite{Handel} (a proof will be given in section \ref{section:Adjacency areas, diagrams and special nice families}).

\begin{prop}\label{prop: adjacency for nice families}
Let $[h;\Or]$ be a Brouwer mapping class. If $\al$ and $\bet$ are two nice families for this class, then they have the same cyclic order up to permutation of arcs of $\bet$ inside the same sub-families of adjacency.
\end{prop}

\subsubsection{Diagram associated to a Brouwer mapping class} 

\begin{figure}[h]
\centering
\vspace{0.2cm}
\includegraphics[scale=0.8]{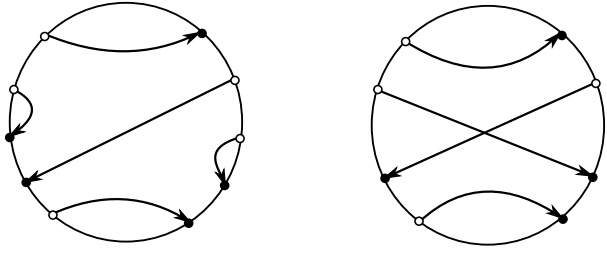}
\caption{Diagrams associated to $[f;\Or]$ and $[g;\Or']$ (examples A and B of section \ref{section: First tools}).}
\label{figu:diag-example}
\end{figure}

Using proposition \ref{prop: adjacency for nice families}, we can associate a diagram to each Brouwer mapping class (see figures \ref{figu:diag-example} and \ref{figu:diagram-example-Handel} for examples):
\begin{enumerate}
\item Let $[h;\Or]$ be a Brouwer mapping class relatively to $r$ orbits. Choose a nice family $(\alpha_i^\pm)_{1\leq i \leq r}$.
\item On the boundary component of a disk, choose one point for each arc of $\al$ in such a way that the $2r$ chosen points respect the cyclic order of $\al$.
\item For every $i$, draw an arrow from the point representing $\alpha_i^-$ to the point representing $\alpha_i^+$. Label this arrow with $i$.
\item Exchange the points in a same sub-family of adjacency if necessary to eliminate as many crossings as possible between the arrows.
\end{enumerate}

\begin{figure}[h]
\centering
\vspace{0.2cm}
\includegraphics[scale=1]{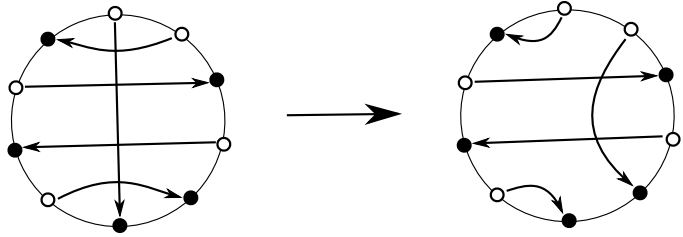}
\caption{An example of step $4$ (this is example $2.9$ of Handel \cite{Handel}).}
\label{figu:diagram-example-Handel}
\end{figure}

We identify two diagrams if they have the same combinatorics. 
\begin{prop}\label{prop: diagram is a conjugacy invariant}
The diagram associated to a Brouwer mapping class is a conjugacy invariant: if two Brouwer mapping classes are conjugated, then they have the same diagram.
\end{prop}

We say that a diagram $\mathcal D$ is \emph{determinant} if, up to conjugacy, there exists only one Brouwer mapping class whose associated diagram is $\mathcal D$. It is a natural question to ask which diagrams are determinant.

For Brouwer mapping classes relatively to one, two and three orbits, the diagram is a total invariant: every diagram with one, two or three arrows is determinant (see \cite{Handel} for one and two orbits and \cite{FLR2013} for three orbits).

However, for Brouwer mapping classes relatively to more than $3$ orbits, the diagram is not a total invariant. For example, consider the flow $f$ and the $f$-free half twist $\mu$ of examples A and B. The Brouwer mapping classes $[\mu^2 f;\Or']$ and $[f;\Or']$ have the same diagram but are not conjugated (we will prove later that they are not conjugated). In section \ref{section:determinant diagrams}, we will give combinatorial conditions on diagrams to prove that some of them are determinant. In section \ref{section: 4 orbits} we will describe all the determinant diagrams for Brouwer mapping classes relatively to $4$ orbits.

\subsubsection{Special nice families}

A \emph{reducing set} is a union of mutually disjoint and non isotopic reducing lines. Any connected component of the complement of a reducing set is said to be a \emph{stable area}. 
In particular, every stable area for $[h;\Or]$ is isotopic to its image by $h$ relatively to $\Or$. If the reducing lines of a given reducing set $\mathcal{R}$ are geodesic, then according to the straightening principle \ref{lemma: straightening principle}, there exists $h'\in [h;\Or]$ such that every stable area $Z$ of the complement of $\mathcal R$ is such that $h'(Z)=Z$. The following proposition will be proved in section \ref{section: nice families disjoint from reducing sets}.

\begin{prop} \label{prop: nice families disjoint from reducing lines}
Let $[h;\Or]$ be a Brouwer mapping class. Let $(\Delta^k)_{k}$ be a reducing set. There exists a nice family $\al$ for $[h;\Or]$ such that for every $k$, for every $i$, $\alpha_i^-$ and $\alpha_i^+$ are homotopically disjoint from $\Delta^k$.
\end{prop}

\subsection{Walls for a Brouwer mapping class (section \ref{section: walls})}
We define a canonical reducing set (the walls) and prove that this set splits the plane into three types of stable areas: stable areas disjoint from $\Or$, translation areas and irreducible areas. In the next section, we will give combinatorial conditions for the existence of irreducible areas.

\subsubsection{Translation areas}

\begin{definition}[Translation area]
Let $[h;\Or]$ be a Brouwer mapping class. We say that a stable area $Z$ of $[h;\Or]$ is a \emph{translation area} if all the orbits of $Z$ are backward adjacent and forward adjacent for $[h;\Or]$.

Moreover, a translation area $Z$ is said to be \emph{maximal} if there exists no translation area $Z'$ non isotopic to $Z$ and such that $Z \subset Z'$.
\end{definition}

Note that every translation area is included in a maximal translation area.

\begin{rem} The orbits of a same translation area for a Brouwer mapping class are represented by arrows which are backward and forward adjacent in the diagram associated to the Brouwer class. However, some arrows which are backward and forward adjacent do not represent orbits of a translation area. For an example, see figure \ref{figu:translation_area}: the Brouwer class that we consider is the product of a flow with a double free half twist between the two arrows that intersect the grey disk of the figure. The two arrows on the bottom of the diagram are backward and forward adjacent but not in the same translation area.
\end{rem}

\begin{figure}[h]
\centering
\includegraphics[scale=0.8]{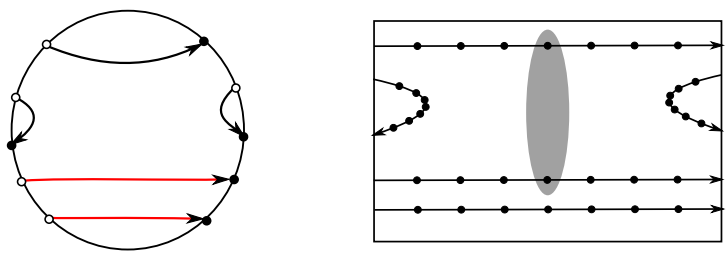}
\caption{Example of a Brouwer class whose diagram has two backward and forward adjacent arrows which are not in the same translation area.}
\label{figu:translation_area}
\end{figure}

The following proposition implies that every Brouwer class has flow streamlines on its translation areas. It will be proved in section \ref{subsection:translation areas}.

\begin{prop} \label{prop:translation area}
If $Z$ is a translation area, every backward (respectively forward) arc of a nice family which is included in $Z$ is also forward (respectively backward).
\end{prop}

\subsubsection{Irreducible areas}

\begin{definition}[Irreducible area]
Let $[h;\Or]$ be a Brouwer mapping class. We say that a stable area $Z$ of $[h;\Or]$ is an \emph{irreducible area} if:
\begin{itemize}
\item $Z$ contains at least $2$ orbits of $\Or$;
\item There is no reducing line of $[h;\Or]$ strictly included in $Z$ (i.e. homotopically disjoint from every boundary component of $Z$ and non isotopic to any of those).
\end{itemize}
\end{definition}

\subsubsection{Walls}

\begin{definition}[Wall]
Let $[h;\Or]$ be a Brouwer mapping class. An isotopy class of a reducing line $\Delta$ for $[h;\Or]$ is called a \emph{wall for $[h;\Or]$} if every reducing line for $[h;\Or]$ is homotopically disjoint from $\Delta$.
\end{definition}

The proof of the following theorem is the aim of section \ref{section: walls}.
\begin{theorem}\label{theorem: walls}
Let $[h;\Or]$ be a Brouwer mapping class. Let $\mathcal{W}$ be a family of pairwise disjoint reducing lines containing exactly one representative of each wall for $[h;\Or]$. If $Z$ is a connected component of $\R^2 - \mathcal{W}$, then exactly one of the followings holds:
\begin{itemize}
\item $Z$ is an irreducible area;
\item $Z$ is a maximal translation area;
\item $Z$ does not intersect $\Or$.
\end{itemize}
\end{theorem}

\begin{rem} Note that:
\begin{itemize}
\item The set of walls is empty if and only if $[h;\Or]$ is a translation class;
\item There exists exactly one wall for $[h;\Or]$ if and only if $[h;\Or]$ is the class of a Reeb flow.
\end{itemize}
\end{rem}

\begin{figure}[h]
\centering
\vspace{0.2cm}
\includegraphics[scale=1]{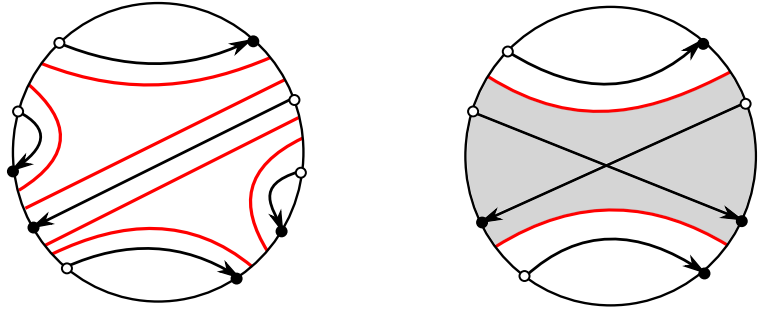}
\caption{Diagrams with walls associated to $[f;\Or]$ and $[g;\Or']$ (examples A and B of section \ref{section: First tools}).}
\label{figu:diagwithwalls}
\end{figure}

Since there exists a nice family disjoint from the walls (according to proposition \ref{prop: nice families disjoint from reducing lines}), it makes sense to add the set of walls on the diagram defined in section \ref{section:Adjacency areas, diagrams and special nice families} (see figure \ref{figu:diagwithwalls} for two examples). We can see the maximal translation areas on this diagram: the backward ends of their arrows are adjacent in the diagram and the forward ends of their arrows are adjacent in the diagram. Consequently, according to theorem \ref{theorem: walls}, we can also see the irreducible areas. To help the reader, we color the irreducible areas in grey. The resulting \emph{diagram with walls} is a conjugacy invariant of the Brouwer mapping class which is more precise than the diagram (without walls), but still not total. In the next section we will give conditions to determine which diagrams with walls are determinant.

\subsection{Determinant diagrams and irreducible areas (section \ref{section:determinant diagrams})}

\subsubsection{Determinant diagrams}
The following propositions motivate the search of necessary combinatoric conditions on diagrams (without walls) for the existence of irreducible areas. They will be proved in section \ref{subsection:determinant diagrams}.

\begin{prop} \label{prop:flow class iff no irreducible area} A Brouwer mapping class $[h;\Or]$ is a flow class if and only if no connected component of the complement of the set of walls for $[h;\Or]$ is an irreducible area.
\end{prop}

\begin{prop}\label{prop:flow implies conjugate}
If two flow classes have the same diagram, then they are conjugated.
\end{prop}

We will deduce that a diagram without crossing arrows is determinant if and only if it does not have any irreducible area:
\begin{prop}\label{coro:determinant diagrams with walls}
A diagram with walls without crossing arrows is determinant if and only if the arrows of every family of arrows included in the same connected component of the complement of the walls are backward adjacent and forward adjacent. 
\end{prop}

\subsubsection{Combinatorics of irreducible areas}

If $[h;\Or]$ is a Brouwer class relatively to $r$ orbits, we denote by $2r'$ the number of adjacency subfamilies of $[h;\Or]$. If $r'=r$, then we say that the orbits of $[h;\Or]$ alternate (in this situation, every adjacency subfamily has only one element). We will prove proposition \ref{prop:combinatoric of an irreducible area} in section \ref{subsection:combinatorics of irreducible areas}.

\begin{prop}[Combinatorics of irreducible areas] \label{prop:combinatoric of an irreducible area}
Let $[h;\Or]$ be a Brouwer mapping class and let $Z$ be an irreducible area for $[h;\Or]$. Then:
\begin{enumerate}
\item The orbits of $Z$ are not all backward adjacent, neither all forward adjacent for $[h;\Or]$;
\item $Z$ has at least two boundary components;
\item The orbits of $[h;\Or \cap Z]$ do not alternate.
\end{enumerate}
\end{prop}

This proposition gives tools to know which diagrams are determinant. In particular, we have the following corollaries, which will be proved in section \ref{subsection:corollaries of irreducible areas}:

\begin{corollary} \label{corollary:r'=2 implies flow class}
Let $[h;\Or]$ be a Brouwer mapping class relatively to $r$ orbits. Denote by $2r'$ the number of adjacency subfamilies of $[h;\Or]$. If $r'= 1,2$ or $r$, then $[h;\Or]$ is a flow class.
\end{corollary} 

\begin{rem} It was proved in \cite{Handel} (for $r'=1$) and \cite{FLR2013} (for $r'=r$) that if $r'=1$ or $r'=r$, then $[h;\Or]$ is a flow class (see proposition $3.1$ and lemma $3.6$ of \cite{FLR2013}).
\end{rem}

\begin{corollary}\label{cor:deux orbites bordantes}
Let $[h;\Or]$ be a Brouwer mapping class relatively to $r$ orbits.
\begin{enumerate}
\item If $r\geq 3$, there exist at least two disjoint and non isotopic reducing lines for $[h;\Or]$;
\item If $r \geq 2$, there exist at least two translation areas for $[h;\Or]$ which have exactly one boundary component;
\item There exists a nice family $\al$ and $j \neq k$ such that:
\begin{itemize}
\item Relatively to $\Or$, $\alpha_j^-$ is isotopic to $\alpha_j^+$ and $\alpha_k^-$ is isotopic to $\alpha_k^+$;
\item In the cyclic order, $\alpha_j^-$ and $\alpha_j^+$ (respectively $\alpha_k^-$ and $\alpha_k^+$) are neighbors.
\end{itemize}
\end{enumerate}
\end{corollary}

The third point of corollary \ref{cor:deux orbites bordantes} can be reformulated as follow.
\begin{corollary}\label{cor:deux orbites bordantes2}
Let $[h;\Or]$ be a Brouwer class relatively to $r\geq 2$ orbits. There exist at least two backward and forward disjoint proper streamlines $T$ and $S$ for $[h;\Or]$.

 Moreover, the orbits which are not in $T$ nor $S$ are included in the same connected component of the complement of $S \cup T$.
\end{corollary}

\subsection{Classification relatively to $4$ orbits (section \ref{section: 4 orbits})}
The aim of section \ref{section: 4 orbits} is to give a complete description of Brouwer mapping classes relatively to $4$ orbits. We first find every diagram with walls which are not determinant. For the Brouwer mapping classes with a non determinant diagram, we define a new conjugacy invariant, \emph{the tangle} (see section \ref{subsection: tangle}). This tangle is an isotopy class of curves on the cylinder with two marked points, up to horizontal twists (see figure \ref{figu:example-invariant2} for an example). We set that the tangle of Brouwer mapping classes without irreducible area is the empty set. We claim that the couple  \texttt{(Diagram with walls, Tangle)} is a total conjugacy invariant for Brouwer mapping classes relatively to $4$ orbits:

\begin{figure}[h]
\centering
\includegraphics[scale=1]{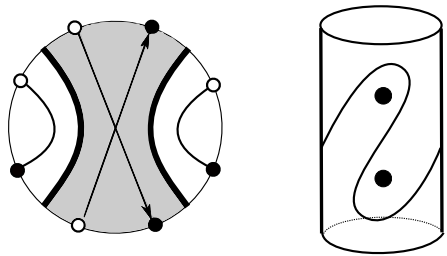}
\caption{Example of a couple (Diagram with walls, Representative of the tangle)}
\label{figu:example-invariant2}
\end{figure}

\begin{theorem}\label{theo:total invariant}
Two Brouwer mapping classes relatively to $4$ orbits are conjugated if and only if they admit the same couple (Diagram with walls, Tangle).
\end{theorem}

We describe which tangles are realized by Brouwer mapping classes and call them \emph{adapted tangles}.

Finally, every couple $ \texttt{(Diagram with walls, Adapted tangle)} $ is realized by:
\begin{itemize}
\item A flow if the diagram is determinant or if the tangle is trivial;
\item A product of a flow and finitely many free half twists if the diagram is not determinant and the tangle is not trivial.
\end{itemize} 
This gives a complete description of the Brouwer mapping classes relatively to $4$ orbits.

\section{Adjacency areas, diagrams and special nice families} \label{section:Adjacency areas, diagrams and special nice families}

\subsection{Adjacency areas}

 Let $[h;\Or]$ be a Brouwer mapping class. Let $(\alpha_i\pm)_{1 \leq i \leq r}$ be a nice family. Let $\{\alpha_{i_1}^+,...,\alpha_{i_n}^+\}$ be a sub-family of adjacency for $[h;\Or]$. For simplicity of notation, we assume that $i_k=k$ for every $1\leq k \leq n$. Choose a complete hyperbolic metric on $\R^2 - \Or$. Let $L$ be a geodesic topological line in $\R^2 - \Or$ such that (see figure \ref{figu:adjacency_area}):
 \begin{itemize}
 \item One connected component of $\R^2 - L$, denoted by $A$, contains an infinite component of each $T_i^+$ for $1\leq i \leq n$;
  \item For every $1 \leq i\leq n$, $L$ intersects $T_i^+:=T^+(\alpha_i^+,h,\Or)$ in exactly one point;
 \item $A$ does not contain any point of $\Or$ outside those $n$ half streamlines.
 \end{itemize}

\begin{figure}[h]
\labellist
\small\hair 2pt	
\pinlabel $A$ at 495 100
\pinlabel $L$ at 545 240
\pinlabel $h(L)_\#$ at 563 214
\endlabellist
\centering
\includegraphics[scale=0.6]{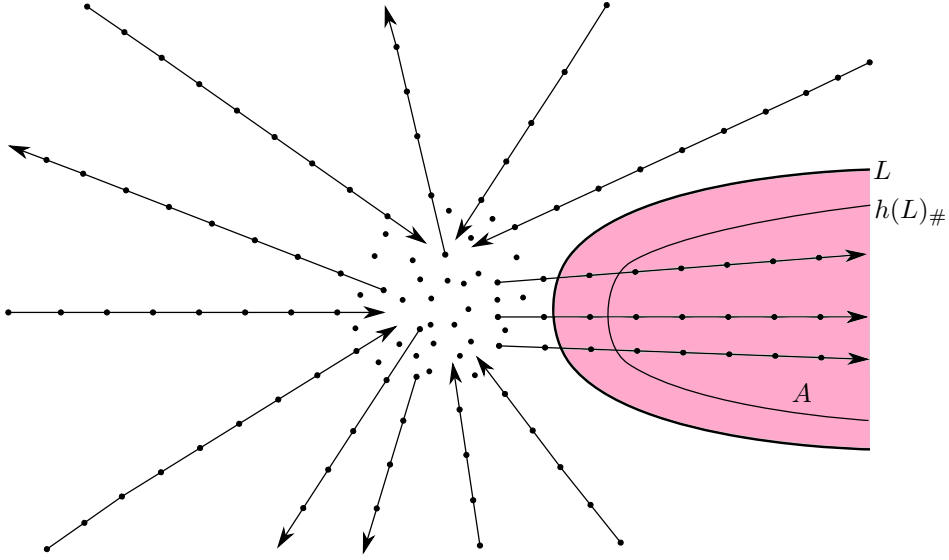}
\caption{Example of a forward adjacency area.}
\label{figu:adjacency_area}
\end{figure}

\begin{definition}[Adjacency area] With the previous notations, we say that $A$ is a \emph{forward adjacency area} for $[h;\Or]$. A forward adjacency area for $[h^{-1};\Or]$ is said to be a \emph{backward adjacency area} for $[h;\Or]$.
\end{definition}

Note that every adjacency area can be obtained with the following construction. \\

\noindent \textit{Construction of an adjacency area.} \label{construction of adjacency areas}
With the previous notations, suppose that the $\alpha_i^+$'s for $i=1...n$ are ordered from $\alpha_1^+$ to $\alpha_n^+$ in the cyclic order of the nice family. Choose one point $x_i \in \Or_i$ in every $T_i^+$. Denote by $L_i$ the unbounded component of $T_i^+ - \{x_i\}$. 
\begin{itemize}
\item If $n=1$, let $\mathcal{U}$ be an open neighborhood of $L_1$ which is homotopic to $L_1$ relatively to $\Or$;
\item If $n>1$: for every $i \leq n-1$, denote by $\gamma_i$ a geodesic arc of $\R^2 - \Or$ which admits $x_i$ and $x_{i+1}$ for endpoints, and such that one connected component of $L_i \cup \gamma_i \cup L_{i+1}$ does not intersect $\Or$. In particular, note that $\{h^n(\gamma_i)_\#\}_{n\geq 0}$ is locally finite. Now consider the line $\tilde L := L_1 \cup \gamma_1 \cup ... \cup \gamma_n \cup L_n$. Let $\mathcal{V}$ be the connected component of $\R^2 -\tilde L$ which contains $L_2$ if $n\geq 3$ and which does not intersect $\Or$ if $n=2$. Let $\mathcal{\U}$ be an open neighborhood of $\mathcal{V}$, isotopic to $\mathcal{V}$ relatively to $\Or$. 
\end{itemize}
Assume that the boundary component $L$ of the closure of $\mathcal{U}$ is geodesic: then $\mathcal{U}$ is an adjacency area. Since there exist pairwise disjoint half homotopy streamlines (theorem \ref{theo: Handel}), there exist pairwise disjoint adjacency areas whose union contain an infinite component of every half orbit (see figure \ref{figu:adjacency_areas}).

\begin{definition}
Choose an adjacency area for every subfamily of adjacency such that the chosen areas are mutually disjoint. Such a family is said to be a \emph{complete family of adjacency areas for $[h;\Or]$}.
\end{definition}

\begin{figure}[h]
\centering
\includegraphics[scale=0.6]{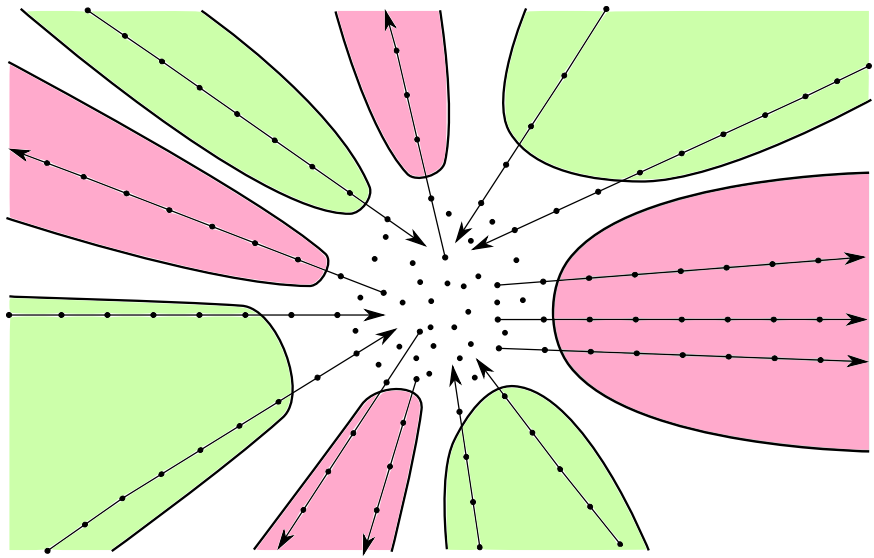}
\caption{Example of a complete family of backward and forward adjacency areas.}
\label{figu:adjacency_areas}
\end{figure}

In particular, the union of the elements of a complete family of adjacency areas contains every point of $\Or$ but a finite number. Moreover, if we consider two complete families of adjacency areas, then there exists a compact $K$ such that these two families are isotopic relatively to $\Or$ in $\R^2 -K$.

The following theorem is essentially due to Handel \cite{Handel}. This statement is a reformulation of proposition $3.1$ of \cite{FLR2013}.
\begin{theorem}[Handel] \label{theo:reducing line disjoint from backward adjacency areas} Let $[h;\Or]$ be a Brouwer class which is not a translation class. Choose a complete family of adjacency. There exists a reducing line disjoint from all the backward adjacency areas of the complete family. \end{theorem}

\begin{proof} There exists a family of generalized homotopy half streamline such that for every backward adjacency area $A$ of the complete family, $\Or \cap A$ is included in one of the backward generalized homotopy half streamlines. Proposition $3.1$ of \cite{FLR2013} gives a reducing line disjoint from all the backward generalized homotopy half streamlines of the family. The result follows.
\end{proof}

\begin{prop} \label{prop:first properties of the adjacency areas}
Let $A$ be an adjacency area for $[h;\Or]$ and let $L$ be its boundary component.
\begin{enumerate}
\item $L$ is a homotopy Brouwer line;
\item The family $(h^k(L)_\#)_{k\in \N}$ is locally finite;
\item (Handel) Let $\beta^+$ be any forward proper arc with endpoints in an orbit of $\Or$ which intersects $A$. There exists $k_0$ such that for every $k>k_0$, $h^k(\beta^+)_\#$ is included in $A$.
\end{enumerate}
\end{prop}

\begin{proof}
Every adjacency area can be seen as in the construction done before. Thus $(1)$ and $(2)$ follow, because $\{h^n(\gamma_i)_\# \}_{n\geq 0}$ is locally finite for every $i$, as well as $\{ h^n(L_i)_\#\}_{n\geq 0}$. The constructed $\mathcal{U}$ is a neighborhood of a ''generalized homotopy streamline`` (see \cite{Handel} and figure \ref{figu:generalized_half_streamlines2}), hence property $(3)$ holds, according to lemma 4.6 of \cite{Handel}.
\end{proof}

\begin{figure}[h] 

\centering
\vspace{0.2cm}
\includegraphics[scale=0.8]{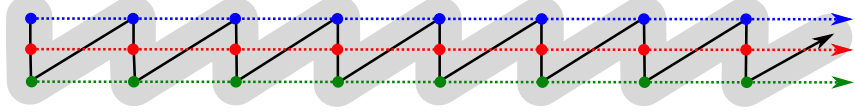}
\caption{Neighborhoods of generalized homotopy half streamlines are homotopic to adjacency areas.}
\label{figu:generalized_half_streamlines2}
\end{figure}

\begin{corollary} \label{cor:translation in adjacency areas}
Let $[h;\Or]$ be a Brouwer mapping class, and let $A$ be an adjacency area for $[h;\Or]$. Then there exists a homeomorphism $\chi$ of the plane, preserving the orientation, such that $\chi h \chi^{-1}$ is isotopic relatively to $\chi(\Or)$ to a homeomorphism which coincides with a translation on $\chi(A)$.
\end{corollary}

\begin{proof}
This follows from $(1)$ and $(2)$ of proposition \ref{prop:first properties of the adjacency areas}.
\end{proof}

\subsection{Diagrams}

\begin{proof}[Proofs of propositions \ref{prop: adjacency for nice families} and \ref{prop: diagram is a conjugacy invariant}.] Let $\al$ and $\bet$ be two nice families for a Brouwer class $[h;\Or]$. According to  $(3)$ of proposition \ref{prop:first properties of the adjacency areas}, for every $i$ there exists an adjacency area $A$ and there exists $k_0 \in \N$ such that for every $k>k_0$, $h^k(\alpha_i^+)$ and $h^k(\beta_i^+)$ are included in $A$. We have a similar result for backward arcs. It follows that $\al$ and $\bet$ have the same cyclic order up to permutation of arcs of $\bet$ inside the same sub-families of adjacency, which is proposition \ref{prop: adjacency for nice families}.

As a corollary we get proposition \ref{prop: diagram is a conjugacy invariant}: if two Brouwer mapping classes are conjugated, then they have the same diagram.
\end{proof}

\subsection{Special nice families} \label{section: nice families disjoint from reducing sets}
The aim of this section is to prove proposition \ref{prop: nice families disjoint from reducing lines}, i.e. that for every reducing set $\mathcal{R}$, there exists a nice family which is disjoint from $\mathcal{R}$.

\subsubsection{Intersections between reducing lines and adjacency areas}

\begin{lemma} \label{lemma: adjacency areas and reducing lines}
Let $[h;\Or]$ be a Brouwer mapping class, with a complete family of adjacency areas. Let $\Delta$ be a reducing line. Then $\Delta$ is isotopic relatively to $\Or$ to a topological line $\Delta'$ which intersects at most two adjacency areas of the family. 

Moreover, for any complete hyperbolic metric on $\R^2 - \Or$, the intersection between the geodesic representative of $\Delta$ for this metric and any adjacency area has at most finitely many connected components.
\end{lemma}

\begin{proof}
Let $\al$ be a nice family for $[h;\Or]$. Choose a complete hyperbolic metric on $\R^2 -\Or$. For every $i$, we denote by $T_i^+$, respectively $T_i^-$, the homotopy half streamline $T^+(\alpha_i^+,h,\Or)$, respectively $T^-(\alpha_i^+,h,\Or)$. According to the straightening principle \ref{lemma: straightening principle}, there exists $h' \in [h;\Or]$ such that $h'(T_i^+) \subset T_i^+$, $T_i^- \subset h'(T_i^-)$ and $h'(\Delta_\#) = \Delta_\#$. \\

\noindent \textbf{Claim 1.} Let $A$ be an adjacency area. We denote by $\partial A$ its boundary component. If $\Delta_\# \cap \partial A$ is non empty, then $\Delta_\# \cap h'^n(\partial A)_\#$ is non empty for every $n \in \Z$.\\

\noindent \textit{Proof of claim 1.} Since $\Delta_\# \cap \partial A$ is non empty, $h'^n(\Delta_\#) \cap h'^n(\partial A)$ is non empty for every $n \in \Z$. Since $\Delta_\#$ and $\partial A$ are geodesic, they are in minimal position. Hence for every $n$, $h'^n(\Delta_\#)$ and $h'^n(\partial A)$ are also in minimal position. It follows that $h'^n(\Delta_\#) \cap h'^n(\partial A)_\#$ is non empty. \qed \\

\noindent \textbf{Claim 2.} Let $A$ be an adjacency area. If $\Delta_\# \cap A$ is non empty, then for every compact subset $K$ of the plane, $(\Delta_\# \cap A) - K$ is non empty. \\

\noindent \textit{Proof of claim 2.}
Assume $A$ is a forward adjacency area (if not, consider $h'^{-1}$ instead of $h$). Let $K$ be any compact subset of the plane. Assume that $\Delta_\# \cap A$ is non empty. Since $(h'^n(\partial A)_\#)_{n\in \N}$ is locally finite (according to $(2)$ of proposition \ref{prop:first properties of the adjacency areas}), there exists $k \in \N$ such that $h'^k(\partial A)_\#$ does not intersect $K$. Since $\partial A$ is a homotopy Brouwer line (according to $(1)$ of proposition \ref{prop:first properties of the adjacency areas}), $h'^k(\partial A)_\#$ is included in $A$. According to claim $1$, $h'^k(\partial A)_\#$ intersects $\Delta_\#$. Claim $2$ follows. \qed \\

Denote by $(A_i)_{1\leq i \leq l}$ the adjacency areas of the chosen complete family. According to claim $1$, if we prove that for some $(n_i)_i \in \Z^l$, $\Delta_\#$ intersects at most two of the $h'^{n_i}(\partial A_i)_\#$'s, then $\Delta_\#$ intersects at most two of the $\partial A_i$'s. Hence, up to replacing $(A_i)_{1\leq i \leq l}$ by $(h'^{n_i}(\partial A_i)_\#)_{1\leq i \leq l}$ such that $\Delta_\#$ intersects at most two of the $h'^{n_i}(\partial A_i)_\#$'s,  we can assume that for every $i,j$, $A_i$ is disjoint from $\alpha_j^-$.\\

\noindent \textbf{Claim 3.} There exists a topological disk $K$ of the plane such that every connected component of $\Delta_\# - K$ intersects at most one adjacency area. \\

\noindent \textit{Proof of claim 3.} Assume by contradiction that for every $K$, one connected component of $\Delta_\# - K$ intersects two adjacency areas. Then there exist two adjacency areas, say $A_i^-$ and $A_j^+$, such that $\Delta_\# \cap A_i^-$ and $\Delta_\# \cap A_j^+$ have infinitely many connected components. Moreover, taking $K$ which intersects every adjacency area of the complete family, we can suppose that $A_i^-$ follows $A_j^+$ in the cyclic order at infinity of the adjacency areas. Hence we can suppose that $A_i^-$ is a backward adjacency area, and $A_j^+$ is a forward adjacency area.

\begin{figure}[h]
\labellist
\small\hair 2pt	
\pinlabel $K$ at 130 130
\pinlabel $\gamma$ at 505 220
\endlabellist
\centering
\vspace{0.2cm}
\includegraphics[scale=0.5]{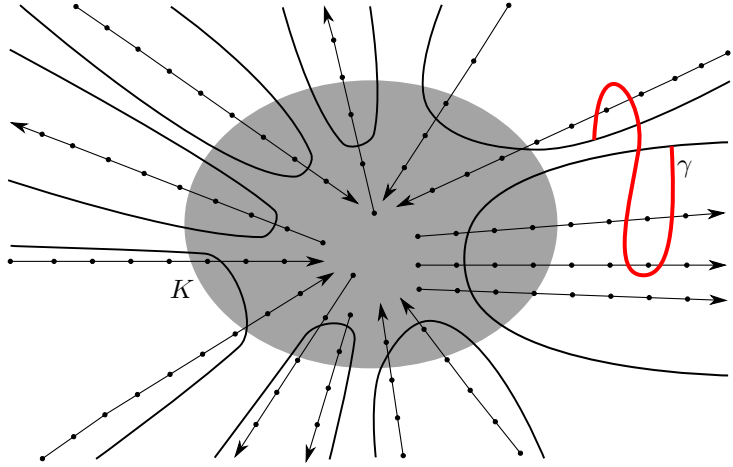}
\vspace{0.2cm}
\caption{Example of a configuration with some $K$ and some $\gamma$.}
\label{figu:gamma}
\end{figure}

It follows that there exists a subsegment $\gamma$ of $\Delta_\#$ such that $\gamma$ is the concatenation of three smaller subsegments $\gamma_1 \star \gamma_2 \star \gamma_3$ such that (see figure \ref{figu:gamma}):
\begin{itemize}
\item $\gamma_1 \subset A_i^-$ and its endpoints are included in $\partial A_i^-$;
\item $\gamma_3 \subset A_j^+$ and its endpoints are included in $\partial A_j^+$;
\item $\gamma_2$ does not intersect $\mathring A_i^-$ not $\mathring A_j^+$, where $\mathring A$ denotes the interior $A - \partial A$ of $A$.
\end{itemize}
Moreover, we can choose $\gamma$ outside any chosen topological disk of the plane: in particular, we choose it disjoint from the $\alpha_i^-$'s.
Since $\Delta_\#$ and the boundary components of the adjacency areas are in minimal position, it follows that:
\begin{itemize}
\item $\gamma_1$ intersects a backward half homotopy streamline $T_i^-$ of $A_i^-$;
\item $\gamma_2$ does not intersect any half homotopy streamline of the family $(T_k^\pm)_k$;
\item $\gamma_3$ intersects a forward half homotopy streamline $T_j^+$ of $A_j^+$.
\end{itemize}

\begin{figure}[h]
\labellist
\small\hair 2pt	
\pinlabel $\delta$ at 90 16
\pinlabel $h'(\delta)$ at 150 16
\endlabellist
\centering
\vspace{0.2cm}
\includegraphics[scale=1]{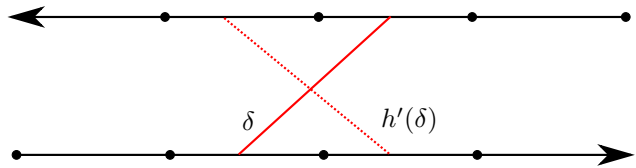}
\vspace{0.2cm}
\caption{Image of $\delta$ under $h'$.}
\label{figu:config-reduc}
\end{figure}

Hence there exists a subsegment $\delta$ of $\gamma$, which contains $\gamma_2$, such that its endpoints are in $T_i^-$ and $T_j^+$ but its interior does not intersect any $T_k^\pm$. Since $h'$ acts as a translation on the $T_k^\pm$'s, it follows that $h'(\delta)$ intersects $\delta$. This gives a contradiction because $\Delta_\#$ is invariant by $h'$ and without self-intersection (see figure \ref{figu:config-reduc}).\qed \\

Let $K$ be a topological disk given by claim $3$. Since $\Delta_\#$ is proper, there are only two unbounded connected components in $\Delta_\# - K$. According to claim $2$, as $\Delta_\#$ and the boundary components of the adjacency areas are in minimal position, every connected component of $\Delta_\# - K$ which intersects an adjacency area is unbounded. Hence $\Delta_\#$ intersects at most two adjacency areas. Moreover, since $\Delta_\#$ and $\partial A_k$ are geodesics, the second part of the lemma follows.
\end{proof}

\begin{figure}[h]
\labellist
\small\hair 2pt	
\pinlabel $\Delta$ at 95 77
\endlabellist
\centering
\vspace{0.2cm}
\includegraphics[scale=1]{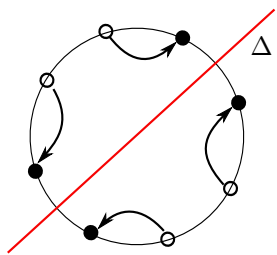}
\vspace{0.2cm}
\caption{Example. Let $[f;\Or]$ be a flow class with this diagram: the reducing line $\Delta$ intersects two forward adjacency areas.}
\label{figu:carre}
\end{figure}

\noindent \textbf{Remark:} If a reducing line intersects two adjacency areas, it does not necessary intersects one backward adjacency area and one forward adjacency area: some reducing lines intersect two adjacency areas of the same type (see figure \ref{figu:carre} for an example).\\

The following lemma will be used in the proof of $(2)$ of proposition \ref{prop:combinatoric of an irreducible area}.
\begin{lemma}\label{lemma:existence of a reducing line almost included in forward adjacency areas}
Let $[h;\Or]$ be a Brouwer class which is not a translation class. Choose a complete family of adjacency areas. There exists a reducing line $\Delta$ such that the intersection of $\Delta$ with the complement of the forward adjacency areas is bounded.
\end{lemma}

\begin{proof} Let $\al$ be a nice family such that each $\alpha_i^-$ intersects the boundary component of a backward adjacency area of the chosen family (such a family exists, up to take iterates of arcs of any nice family). Since $[h;\Or]$ is not a translation class, according to theorem \ref{theo:reducing line disjoint from backward adjacency areas}, there exists a reducing line $\Delta$ which is disjoint from every backward adjacency area. Choose a complete hyperbolic metric on $\R^2 - \Or$. Suppose $\Delta$ is geodesic. Let $K$ be a topological compact disk of the plane such that:
\begin{itemize}
\item The boundary component $\partial K$ of $K$ is geodesic;
\item $K$ intersects every adjacency areas of the family;
\item The union $\mathcal{A}$ of $K$ with all the adjacency areas of the family contains $\Or$.
\end{itemize}
Since $\Delta$ is proper, it intersects $\partial K$ only finitely many times. Denote by $\Delta^+$ and $\Delta^-$ the two unbounded components of $\Delta - K$. Since $\Delta^+$ and $\Delta^-$ are disjoint from the backward adjacency areas, we can isotopy each of them if necessary to include $\Delta$ in a forward adjacency area.
\end{proof}

\subsubsection{Proof of proposition \ref{prop: nice families disjoint from reducing lines}}
\begin{prop*} \emph{\textbf{\ref{prop: nice families disjoint from reducing lines}}}
Let $[h;\Or]$ be a Brouwer mapping class. Let $(\Delta^k)_{k}$ be a reducing set. There exists a nice family $\al$ for $[h;\Or]$ such that for every $k$, for every $i$, $\alpha_i^-$ and $\alpha_i^+$ are homotopically disjoint from $\Delta^k$.
\end{prop*}

\begin{proof}
Choose a complete hyperbolic metric on $\R^2 - \Or$. As usual, we denote by $L_\#$ the geodesic representative of any line or arc $L$. Choose a complete family of adjacency areas for $[h;\Or]$. In every adjacency area, we will construct pairwise disjoint backward or forward proper arcs for each orbit which intersects the area, such that every constructed arc is disjoint from the reducing set. By iterating those arcs so that the backward and forward arcs of a given orbit have the same endpoints, we get the needed nice family. Let $A$ be an adjacency area of the complete family of adjacency. Suppose $A$ is a forward adjacency area (if not, consider $h^{-1}$). Note that $\partial A$ is geodesic (by definition). For simplicity of notations, we assume that $\Delta^1,...,\Delta^N$ are the reducing lines of $(\Delta^k)_k$ intersecting $A$. We assume that these reducing lines are geodesic. According to lemma \ref{lemma: adjacency areas and reducing lines}, each of them has an unbounded connected component included in $A$. Denote by $\Delta^k_+$ this unbounded component for $\Delta^k$. Suppose that $\Or_1,...,\Or_m$ are the orbits of $\Or$ which intersect $A$. We will find $m$ mutually disjoint forward proper arcs $(\beta_i^+)_{1\leq i \leq m}$ for $[h;\Or]$, included in $A$ and homotopically disjoint from $\Delta^k$ for every $k$.

Applying the straightening principle \ref{lemma: straightening principle} to the families $(\Delta^k)_k$ and $(h^n(\partial A))_{n\geq 0}$, we can find $h' \in [h;\Or]$ such that $h'(\Delta^k)= \Delta^k$ for every $k$ and $(h')^n(\partial A)=h^n(\partial A)_\#$ for every $n$. Note that $h'$ is conjugate to the translation on $A$ (according to corollary \ref{cor:translation in adjacency areas}).

Let $\C$ be the quotient of $A$ by $h'$. Denote by $\pi$ the quotient map. In particular:
\begin{itemize}
\item $\C$ is a topological cylinder;
\item $\pi(\Or \cap A)$ is a set of $m$ points $\hat x_1,...,\hat x_m$ on $\C$;
\item For every $k$, $\pi (\Delta^k \cap A)=\pi (\Delta^k_+)$ is a separating topological circle of $\C -\{\hat x_1,...,\hat x_m\}$;
\item The $\pi (\Delta^k \cap A)$'s are mutually disjoint.
\end{itemize} 
For simplicity of notation, we see $\C$ as a vertical cylinder. There exists a homeomorphism $\phi$ of $\C$ sending each $\pi (\Delta^k \cap A)$ on a horizontal circle $\gamma_k$ and the family $(\hat x_i)_i$ on points of $\C$ with mutually distinct heights. Now for every $1 \leq i \leq m$, consider the horizontal circle $\gamma'_i$ containing $\phi(\hat x_i)$. Every $\gamma'_i$ is disjoint from $\phi \pi (\Delta^k \cap A)$ for every $k$. Hence $(\phi^{-1}(\gamma'_i))_i$ is a family of mutually disjoint curves disjoint from $\pi (\Delta^k \cap A)$ for every $k$. For every $i$, choose a lift $\beta_i^+$ of $\phi^{-1}(\gamma'_i)$ included in $A$, i.e. an arc included in $A$ and such that $\pi(\beta_i^+)=\phi^{-1}(\gamma'_i)$. Such a $\beta_i^+$ is a translation arc. Since $\{h'^n(\partial A)_\# \}_{n\geq 0}$ is locally finite (proposition \ref{prop:first properties of the adjacency areas}), the $\beta_i^+$'s are forward proper . They are disjoint from the $\Delta^k$'s, as wanted. \end{proof}

\section{Walls for a Brouwer mapping class} \label{section: walls}

The main aim of this section is to prove that the set of walls splits $\R^2$ into translation areas, irreducible areas and stable areas that do not intersect $\Or$ (theorem \ref{theorem: walls}).

\subsection{Translation areas} \label{subsection:translation areas}

\begin{lemma} \label{lemma:fitted family} Let $Z$ be a stable area of a Brouwer class such that all the orbits of $Z$ are forward adjacent. Then every backward proper arc included in $Z$ is forward proper.
\end{lemma}

\begin{proof} Let $[h;\Or]$ be a Brouwer class with a complete family of adjacency areas. Let $Z$ be a stable area for $[h;\Or]$ which intersects only one forward adjacency area. Denote by $A$ this adjacency area. Up to replace $h$ by $h'\in [h;\Or]$, according to the straightening principle \ref{lemma: straightening principle}, we can assume that $h(Z)=Z$. Let $\al$ be a nice family for $[h;\Or]$ disjoint from the boundary components of $Z$ (such a family exists, according to proposition \ref{prop: nice families disjoint from reducing lines}). We prove the following claim, which is a consequence of theorem $5.5$ of Handel \cite{Handel}. \\

\noindent \textbf{Claim.} For every $i$ such that $\alpha_i^-$ is in $Z$, there exists $n$ such that $h^n(\alpha_i^-)_\#$ is in $A$.\\

We use the definitions and notations of Handel \cite{Handel}, section $5$ ("Fitted family"). We denote by $W$ the Brouwer subsurface $\R^2 - \cup_k A^+_k$, where $\cup_k A^+_k$ is the union of forward adjacency areas. If $\alpha_i^- \in Z$ is such that for every $n\geq 0$, $h^n(\alpha_i^-)_\# \cap W \neq \emptyset$, then there exists a fitted family $T \subset RH(W,\delta_+W)$ such that:
\begin{itemize}
\item Every $s \in T$ is included in $Z$ (because the elements of $T$ are subsegments of iterates of $\alpha_i^-$, which is included in $Z$, and we have $h(Z)=Z)$;
\item There exists $t \in T$ whose endpoints lie on distinct components of $\delta_+ W$ (this is theorem $(5.5.c)$ of \cite{Handel}).
\end{itemize}  
Since $\delta_+W \cap Z$ has only one component (the boundary component of $A$), the last point does not hold, and thus every $\alpha_i^- \in Z$ is such that for every sufficiently big $n$, $h^n(\alpha_i^-)$ is homotopically included in $A$. It follows that every $\alpha_i^- \in Z$ is forward proper.
\end{proof}

\begin{prop*} \emph{\textbf{\ref{prop:translation area}}}
If $Z$ is a translation area, every backward (respectively forward) arc of a nice family which is included in $Z$ is also forward (respectively backward).
\end{prop*} 

\begin{proof}
By definition, all the orbits of a translation area are backward adjacent and forward adjacent. The result is a consequence of lemma \ref{lemma:fitted family} applied to the Brouwer class, and respectively to its inverse.
\end{proof}

\subsection{Intersections between reducing lines}

\begin{lemma}[Intersection of two reducing lines] \label{lemma: Intersection of two reducing lines}
Let $[h;\Or]$ be a Brouwer mapping class and let $\Delta$ and $\Delta'$ be two reducing lines for $[h;\Or]$. We assume that $\Delta$ and $\Delta'$ are in minimal position. Then one of the following situation holds.
\begin{enumerate}
\item $\Delta \cap \Delta' = \emptyset$.
\item $\Delta \cap \Delta'$ contains exactly one point.
\item $\Delta \cap \Delta'$ is an infinite set.
\end{enumerate}
\end{lemma}

\begin{proof} 
Choose a complete hyperbolic metric on $\R^2 - \Or$. Taking their images by isotopies if necessary, we can suppose that $\Delta$ and $\Delta'$ are geodesic. We use the straightening principle \ref{lemma: straightening principle} to find a homeomorphism $h' \in [h;\Or]$ such that $h'$ preserves $\Delta$ and $\Delta'$.

If $\Delta \cap \Delta'$ contains more than one point, then there exists a bounded connected component of $\R^2 - (\Delta \cup \Delta')$ which contains one point $x \in \Or$. Denote by $C_x$ this component. Then $h'(C_x)$ is a bounded component of $\R^2-(\Delta \cup \Delta')$ different from $C_x$. Indeed, if $h'(C_x)$ coincides with $C_x$, then $h'^n(C_x)=C_x$ for every $n\geq 0$. Hence $\{h'^n(x)\}_{n\geq 0}$ is included in $C_x$. This is not possible because $h'^n(x)=h^n(x)$ for every $n$: since$h$ is a Brouwer homeomorphism, $\{h^n(x)\}_{n\geq 0}$ is not bounded (proposition $3.5$ of \cite{Guillou}).

For the same reasons, for every $k < n \in \N$, $h^n(C_x)$ is disjoint from $h^k(C_x)$. Thus there exist infinitely many bounded connected component of $\R^2 - (\Delta \cup \Delta')$. Hence $\Delta$ and $\Delta'$ have infinitely many intersection points.
\end{proof}

\begin{lemma}[Intersection between a reducing line and a translation area] \label{lemma: Intersection between a reducing line and a translation area}
Let $[h;\Or]$ be a Brouwer mapping class and let $Z$ be a translation area for $[h;\Or]$. Let $\Delta$ be a reducing line. If there exists a boundary component $L$ of $Z$ such that $L$ and $\Delta$ are not homotopically disjoint, then $\Delta \cap L$ is an infinite set.
\end{lemma}

\begin{proof}The line $L$ is a reducing line, hence it is isotopic to its image by $h$. Choose a complete hyperbolic metric on $\R^2 - \Or$. Suppose that $L$ and $\Delta$ are geodesic. For every orbit $\Or_i$ of $\Or$ included in $Z$, we choose a homotopic proper translation arc $\alpha_i$ included in $Z$ such that the $\alpha_i$'s are mutually disjoint (given by proposition \ref{prop: nice families disjoint from reducing lines}). If $\alpha$ is one of these homotopy translation arcs, we denote by $T_\alpha$ the proper streamline $\cup_{n \in \Z} h^n(\alpha)_\#$. Since $L$ and $\Delta$ are not homotopically disjoint, there exists $\alpha_i$ such that $T_{\alpha_i} \cap \Delta \neq \emptyset$. Since $T_{\alpha_j}$ and $L$ are disjoint for every $j$, the straightening principle \ref{lemma: straightening principle} gives us a homeomorphism $h' \in [h;\Or]$ which preserves $L$, $\Delta$ and $T_{\alpha_j}$ for every $i$.

Suppose that $\Delta \cap L$ is not infinite. According to lemma \ref{lemma: Intersection of two reducing lines}, since this intersection is not empty, it contains only one point, say $x$. In particular, we have $h'(x)=x$. Choose an orientation on $\Delta$. Let $y$ be the first intersection point between $\Delta$ and $T_{\alpha_i}$ after $x$ on $\Delta$. Denote by $[xy]$ the segment of $\Delta$ between $x$ and $y$. We have $h'(y) \in T_{\alpha_i}$ and $h'(]xy[) \cap (L \cup T_{\alpha_i}) = \emptyset$, hence $y=h'(y)$. This gives a contradiction because $y$ is contained in a proper translation arc for $h'$.
\end{proof}

\begin{lemma}[Intersection between reducing lines: infinite set case] \label{lemma: Intersection between reducing lines: infinite set case}
Let $[h;\Or]$ be a Brouwer mapping class and let $\Delta$ and $\Delta'$ be two reducing lines for $[h;\Or]$. We assume that $\Delta$ and $\Delta'$ are in minimal position. 

\noindent If $\Delta \cap \Delta'$ is an infinite set, then $\Delta \cup \Delta'$ is included in a translation area.
\end{lemma}

\begin{proof} Choose a complete hyperbolic metric on $\R^2 - \Or$. Isotopying if necessary, we can assume that $\Delta$ and $\Delta'$ are geodesic. The straightening principle \ref{lemma: straightening principle} gives us $h' \in [h;\Or]$ which preserves $\Delta$ and $\Delta'$.
Choose a complete family of mutually disjoint adjacency areas for $[h;\Or]$. Choose a bounded connected component $C_x$ of the complement of $\Delta \cup \Delta'$ which contains a point $x$ of an orbit $\Or_i$ of $\Or$. Denote by $A^-$ and $A^+$ the backward and forward adjacency areas of the chosen complete family which are intersected by $\Or_i$. As shown in the proof of lemma \ref{lemma: Intersection of two reducing lines}, $h'(C_x)$ is a bounded connected component of $R^2 - (\Delta \cup \Delta')$ different from $C_x$. Hence every path from $x$ to $h'(x)$ intersects $\Delta \cup \Delta'$.

According to proposition \ref{prop: nice families disjoint from reducing lines}, there exists a forward proper arc $\alpha^+$ for $\Or_i$ which joints $x$ to $h'(x)$ and which is disjoint from $\Delta'$. Denote by $T^+(\alpha^+)$ the forward half streamline $\cup_{n\geq 0} h^n(\alpha^+)_\#$. Note that $T^+(\alpha^+)$ is disjoint from $\Delta'$. According to proposition \ref{prop:first properties of the adjacency areas}, there exists an unbounded component of $T^+(\alpha^+)$ which is included in $A^+$.
Since $T^+(\alpha^+)$ is proper and disjoint from $\Delta'$, the straightening principle \ref{lemma: straightening principle} give us $h_1 \in [h;\Or]$ which preserves $T^+(\alpha^+)$, $\Delta$ and $\Delta'$. The arc $\alpha^+$ intersects $\Delta$, hence $h_1^n(\alpha^+)$ also intersects $\Delta$ for every $n\in \N$. It follows that $\Delta$ intersects $A^+$. 

The same argument with a backward proper arc $\alpha^-$ disjoint from $\Delta'$ shows that $\Delta$ also intersects $A^-$. According to lemma \ref{lemma: adjacency areas and reducing lines}, every geodesic reducing line intersects at most two adjacency areas: for $\Delta$, this adjacency areas are $A^-$ and $A^+$. Interchanging $\Delta$ and $\Delta'$, we get by the same arguments that $\Delta'$ also intersects $A^-$ and $A^+$.

We choose an orientation on $\Delta$ and $\Delta'$ such that they are oriented from $A^-$ to $A^+$. There exists an unbounded connected component of the complementary of $\Delta \cup \Delta'$ which is on the left of $\Delta$ and $\Delta'$. We denote by $L_l$ its boundary component. Likewise, we denote by $L_r$ the boundary component of the unbounded connected component of the complementary of $\Delta \cup \Delta'$ which is on the right of $\Delta$ and $\Delta'$. The two lines $L_l$ and $L_r$ are proper, because they are unions of segments of two topological lines. Moreover, they are preserved by $h'$.

Now we have the following cases, depending on the positions of the orbits:
\begin{itemize}
\item If $L_r$ and $L_l$ split the set of orbits, then their geodesic representatives $(L_r)_\#$ and $(L_g)_\#$ are disjoint reducing lines which intersect the same adjacency areas. Denote by $Z$ the stable area bounded by $(L_r)_\#$ and $(L_g)_\#$. Thus $Z$ intersects only two adjacency areas, $A^-$ and $A^+$, hence $Z$ is a translation area, which contains $\Delta$ and $\Delta'$;
\item If none of $L_r$ and $L_l$ split the set of orbits, then there exist only two adjacency areas. Hence $[h;\Or]$ is a translation, and the whole plane is a translation area;
\item If only one of $L_r$ and $L_l$ splits the set of orbits, $L_r$ for example, then $L_r$ is a reducing line for $[h;\Or]$. The connected component of $\R^2 - (L_r)_\#$ which contains $\Delta$ and $\Delta'$ is a translation area.
\end{itemize}
\end{proof}

\begin{lemma} [Intersection between reducing lines: case with exactly one point] \label{lemma: Intersection between reducing lines: case with exactly one point}
Let $[h;\Or]$ be a Brouwer mapping class and let $\Delta$ and $\Delta'$ be two reducing lines for $[h;\Or]$. We assume that $\Delta$ and $\Delta'$ are in minimal position. 

\noindent If $\Delta \cap \Delta'$ contains exactly one point, then there exist four reducing lines that are mutually non isotopic and homotopically disjoint and disjoint from $\Delta$ and $\Delta'$. 

Moreover, if we denote by $p$ the intersection point and by $\Delta_1$ and $\Delta_2$, respectively $\Delta'_1$ and $\Delta'_2$, the two half-lines of $\Delta-\{p\}$, respectively of $\Delta'-\{p\}$, these four reducing lines are isotopic relatively to $\Or$ to $\Delta_1 \cup \Delta_2$, $\Delta'_1 \cup \Delta_2$, $\Delta_1 \cup \Delta'_2$ and $\Delta'_1 \cup \Delta'_2$.
\end{lemma}

\begin{figure}[h]
\centering
\includegraphics[scale=0.8]{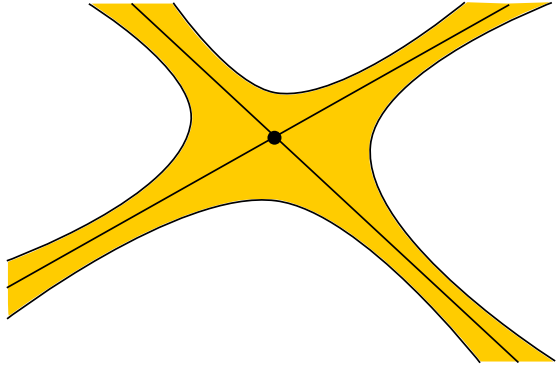}
\caption{Neighborhood of $\Delta \cup \Delta'$.}
\label{figu:prop_1intersection}
\end{figure}

\begin{proof} Choose a complete hyperbolic metric on $\R^2 - \Or$. Isotopying if necessary, we can assume that $\Delta$ and $\Delta'$ are geodesic. The straightening principle \ref{lemma: straightening principle} gives us $h' \in [h;\Or]$ which preserves $\Delta$ and $\Delta'$.

Consider a proper open neighborhood $\mathcal{U}$ of $\Delta \cup \Delta'$ which does not intersect $\Or$ and which is isotopic to $\Delta \cup \Delta'$ relatively to $\Or$ (as in figure \ref{figu:prop_1intersection}). The complement of $\mathcal{U}$ has $4$ connected components. Each of them contains at least one orbit, because $\Delta \cup \Delta'$ are in minimal position. Hence the boundary component of the closure of $\mathcal{U}$ in $\R^2$ is a union of $4$ reducing lines, mutually non isotopic, mutually disjoint, and each of them is disjoint from $\Delta \cup \Delta'$.
\end{proof}

\subsection{Study of the set of walls}

\subsubsection{Maximal translation areas}
We show that there exist finitely many maximal translation areas (proposition \ref{prop: finite number of max. translation areas}), and that the boundary components of this areas are walls (proposition \ref{prop: boundary components of max. translation areas}).

\begin{prop} \label{prop: finite number of max. translation areas}
Let $[h;\Or]$ be a Brouwer mapping class. Up to isotopy, there exist finitely many maximal translation areas. Moreover, they are mutually homotopically disjoint.
\end{prop}

\begin{rem} The statement \ref{prop: finite number of max. translation areas} is generally false if we replace \emph{maximal translation area} by \emph{translation area}. Indeed, if a translation area $Z$ of a Brouwer class $[h;\Or]$ contains at least two orbits, then there are infinitely many non isotopic sub-translation areas included in $Z$. See figure \ref{figu:reduction-lines_translation} for examples of reducing lines for the translation.
\end{rem}

\begin{figure}[h]
\centering
\includegraphics[scale=1]{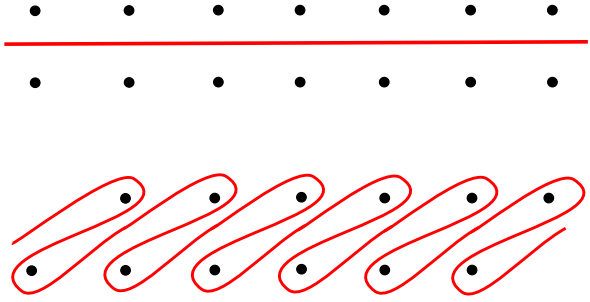}
\vspace{0.2cm}
\caption{Example of two reducing lines for the translation: the complement are translation areas.}
\label{figu:reduction-lines_translation}
\end{figure}

\begin{proof}[Proof of proposition \ref{prop: finite number of max. translation areas}]
We choose a complete hyperbolic metric on $\R^2 - \Or$ and a complete family of adjacency areas for $[h;\Or]$. Let $Z$ and $Z'$ be two non homotopic maximal translation areas. Denote by $A^-,A^+$, respectively $B^-,B^+$, the adjacency areas intersected by $Z$, respectively $Z'$. We claim that:
\begin{enumerate}
\item The boundary components of $Z$ and $Z'$ are homotopically disjoint;
\item No boundary component of $Z$ is included in $Z'$.
\end{enumerate}
\textit{Proof of (1).} If a boundary component $L$ of $Z$ intersects a boundary component $L'$ of $Z'$, then $L$ and $L'$ have infinitely many intersection points (according to lemma \ref{lemma: Intersection between a reducing line and a translation area}). Hence $L$ and $L'$ are included in the same translation area (according to lemma \ref{lemma: Intersection between reducing lines: infinite set case}). Denote by $Z''$ this translation area, and by $C^-$ and $C^+$ the two adjacency areas intersected by $Z''$. Since $L$ and $L'$ are reducing lines, they intersect at most two adjacency areas: $C^-$ and $C^+$. The cyclic order of the adjacency areas at infinity is such that backward areas and forward areas alternate (by definition of adjacency areas). It follows that $A^- =B^-=C^-$ and $A^+=B^+=C^+$. 

According to lemma \ref{lemma: straightening principle}, there exists $h' \in [h;\Or]$ such that $h'(Z)=Z$ and $h'(Z')=Z'$. It follows that the boundary components of $Z\cup Z'$ are reducing lines, hence $Z\cup Z'$ is a stable area. Since $Z\cup Z'$ intersects only two adjacency areas ($C^-$ and $C^+$), it is a translation area. This gives a contradiction with the maximality of $Z$ and $Z'$ as translation areas.\\

\noindent \textit{Proof of (2).} If a boundary component $L$ of $Z$ is included in $Z'$, then there exists an orbit $\Or_i$ included in $Z \cap Z'$. Hence again $A^- =B^-$ and $A^+=B^+$, which gives a contradiction.\\

We complete the proof of proposition \ref{prop: finite number of max. translation areas}. Every maximal translation area contains at least one orbit. Since there are finitely many orbits and since this areas are mutually disjoint, there are finitely many maximal translation areas.
\end{proof}

\begin{prop}\label{prop: boundary components of max. translation areas}
Let $[h;\Or]$ be a Brouwer mapping class and let $Z$ be a maximal translation area. Each isotopy class of a boundary component of $Z$ is a wall of $[h;\Or]$.
\end{prop}

\begin{proof}
Let $L$ be a boundary component of a maximal translation area $Z$. We need to show that if $\Delta$ is a reducing line which is non isotopic to $L$, then $\Delta \cap L =\emptyset$. According to lemma \ref{lemma: Intersection between a reducing line and a translation area}, if a reducing line $\Delta$ intersects $L$ then $\Delta \cap L$ is an infinite set. According to lemma \ref{lemma: Intersection between reducing lines: infinite set case}, it follows that $L \cup \Delta$ is included in a translation area. Since maximal translation areas are mutually disjoint (according to proposition \ref{prop: finite number of max. translation areas}), $\Delta$ is included in $Z$ (which contains $L$): this is impossible, hence every reducing line $\Delta$ is homotopically disjoint from $L$.
\end{proof}

\subsubsection{Outside the translation areas}

This subsection completes the picture: there are finitely many maximal translation areas, mutually homotopically disjoint, and outside those areas there are only finitely many geodesic reducing lines, which intersect mutually in zero or one point.

\begin{lemma}\label{lemma: disjoint reducing lines}
Let $[h;\Or]$ be a Brouwer mapping class. Let $\Delta$ and $\Delta'$ be two disjoint reducing lines. If $\Delta$ and $\Delta'$ split the orbits into the same two subfamilies, then $\Delta$ and $\Delta'$ are isotopic.
\end{lemma}

\begin{proof}
The set $\Delta \cup \Delta'$ splits the plane into three connected components. One of them (the one in the middle) is disjoint from $\Or$.
\end{proof}

\begin{prop} \label{prop: outside translation areas}
Let $[h;\Or]$ be a Brouwer mapping class. Outside the maximal translation areas, there exists only finitely many non isotopic reducing lines.
\end{prop}

\begin{proof}
We prove that there exists only finitely many non isotopic reducing lines in every connected components of the complement of the union of the translation areas. If two such reducing lines are not homotopically disjoint, then they have only one intersection point (according to lemmas \ref{lemma: Intersection of two reducing lines} and \ref{lemma: Intersection between reducing lines: infinite set case}). Hence they do not split the orbits in the same subfamilies. This remark together with lemma \ref{lemma: disjoint reducing lines} imply that if we choose a partition of the orbits in the chosen component into two subfamilies, there exists at most one reducing line included in the complement which splits the orbits into the same partition. Since there exist only finitely many different partitions of the orbits into two subfamilies, there are only finitely many isotopy classes of reducing lines. 
\end{proof}

\subsubsection{Proof of theorem \ref{theorem: walls}}
\begin{theorem*}\emph{\textbf{\ref{theorem: walls}}}
Let $[h;\Or]$ be a Brouwer mapping class. Let $\mathcal{W}$ be a family of pairwise disjoint reducing lines containing exactly one representative of each wall for $[h;\Or]$. If $Z$ is a connected component of $\R^2 - \mathcal{W}$, then exactly one of the followings holds:
\begin{itemize}
\item $Z$ is an irreducible area;
\item $Z$ is a maximal translation area;
\item $Z$ does not intersect $\Or$.
\end{itemize}
\end{theorem*}

\noindent \textit{Proof.} Choose a complete hyperbolic metric on $\R^2 - \Or$. Up to isotopying relatively to $\Or$, we can assume that the elements of $\mathcal W$ are geodesic. According to proposition \ref{prop: boundary components of max. translation areas}, the isotopy classes of boundary components of maximal translation areas are walls. Let $Z$ be a connected component of $\R^2 - \mathcal{W}$ which is not a translation area. Suppose that $Z$ is not irreducible. Then there exists a reducing line included in $Z$ which is not homotopic to any boundary component of $Z$. According to proposition \ref{prop: outside translation areas}, $Z$ contains a finite number of mutually non isotopic reducing lines. Since they are not walls, each of them intersects another one. In particular there are at least two reducing lines included in $Z$ and not homotopic to any boundary component of $Z$.

\begin{figure}[h]
\centering
\labellist
\small\hair 2pt
\pinlabel $\delta_1$ at 102 15
\pinlabel $\delta_2$ at 93 60
\pinlabel $\delta_3$ at 130 78
\pinlabel $\delta_4$ at 175 84
\pinlabel $\delta_5$ at 207 76
\pinlabel $\delta_6$ at 203 43
\pinlabel $L$ at 182 5
\endlabellist
\includegraphics[scale=0.9]{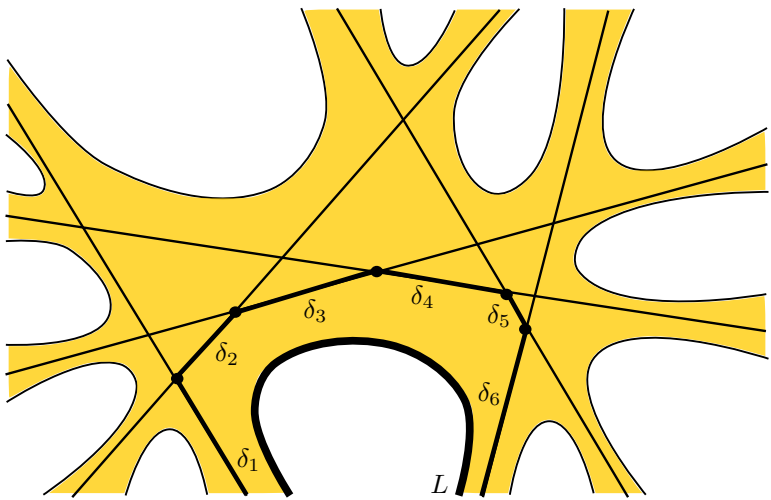}
\caption{Example of $\mathcal{U}$, $\tilde{\mathcal{U}}$ and a boundary component $L$ isotopic to $\delta_1\cup ... \cup \delta_6$.}
\label{figu:theorem_walls}
\end{figure}

 Denote by $\mathcal{U}$ the finite union of geodesic reducing lines included in $Z$ and not homotopic to any boundary component of $Z$. Denote by $\tilde{\mathcal{U}}$ a tubular neighborhood of $\mathcal{U}$ which does not intersect $\Or$ (see figure \ref{figu:theorem_walls}). Choose $\tilde{\mathcal{U}}$ such that the boundary components its closure are geodesic. Denote by $L$ one of them. We claim the following:\\

\noindent \textbf{Claim.} The line $L$ is a reducing line for $[h;\Or]$.\\

\noindent \textit{Proof of the claim.} The line $L$ is homotopic to a union $L'$ of a finite number of segments included in distinct reducing lines (see figure \ref{figu:theorem_walls_claim}). The number of segments is finite because of the following properties.
\begin{enumerate}
\item The area $Z$ is homotopically disjoint from the translation areas;
\item Up to isotopy there are only finitely many reducing lines outside the translation areas (according to proposition \ref{prop: outside translation areas});
\item If two reducing lines outside the translation areas intersect, then their intersection is exactly one point: according to lemma \ref{lemma: Intersection of two reducing lines}, this intersection is either one point or infinite, and according to lemma \ref{lemma: Intersection between reducing lines: infinite set case}, if the intersection is infinite then the reducing lines are included in a translation area.
\end{enumerate} 
 Denote by $n$ this number, and by $\delta_1 \cup ... \cup \delta_n$ the segments whose union is $L'$. We assume that the $\delta_i$'s are in this order on $L$ (as in figure \ref{figu:theorem_walls}). For every $i$, denote by $\Delta_i$ a reducing line of $\mathcal U$ which contains $\delta_i$. Denote by $L_1$ the line obtained as the union of $\delta_1$ and the half-line of $\Delta_2$ whose endpoint is the intersection point between $\delta_1$ and $\delta_2$ and which contains $\delta_2$. According to lemma \ref{lemma: Intersection between reducing lines: case with exactly one point}, $L_1$ is a reducing line for $[h;\Or]$. For every $2 \leq i \leq n-1$, denote inductively by $L_i$ the line obtained as the union of the half-line $L_{i-1}$ which contains $\delta_1$ and the half-line $\Delta_{i+1}$ which contains $\delta_{i+1}$ (both half lines have the intersection point between $\delta_i$ and $\delta_{i+1}$ for endpoint). Applying lemma \ref{lemma: Intersection between reducing lines: case with exactly one point} inductively, we see that $L_i$ is a reducing line for every $i$. Hence $L'=L_{n-1}$ is a reducing line for $[h;\Or]$.\qed \\

\begin{figure}[h]
\labellist
\small\hair 2pt
\pinlabel $L'$ at 70 115
\pinlabel $L_1$ at 230 115
\pinlabel $L_2$ at 390 115
\pinlabel $L_3$ at 70 -15
\pinlabel $L_4$ at 230 -15
\pinlabel $L_5$ at 390 -15
\endlabellist
\centering
\includegraphics[scale=0.8]{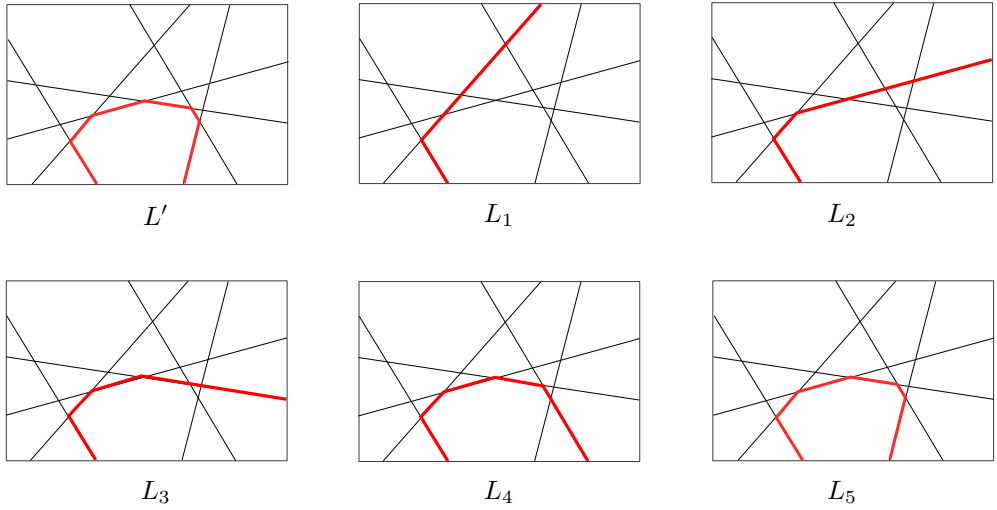}
\vspace{0.6cm}
\caption{Proof of the claim in the case of the example.}
\label{figu:theorem_walls_claim}
\end{figure}

\noindent \textit{End of the proof of theorem \ref{theorem: walls}.} Since $L$ is in $Z$ but not in $\mathcal{U}$, it is homotopic to a boundary component of $Z$ (by definition of $\mathcal{U}$). Hence $\partial \tilde{\mathcal{U}}$ is included in $\partial Z$, thus $Z$ is included in $\tilde{\mathcal{U}}$ (because they both are intersections of half topological planes). It follows that $Z$ is isotopic to $\tilde{\mathcal{U}}$ relatively to $\Or$, hence $Z$ does not intersect $\Or$. \qed\\

\begin{rem}
Note that a stable area is irreducible if and only if it is an irreducible area of the complement of the set of walls. In particular, the isotopy classes of the boundary components of any irreducible area are walls.
\end{rem}

\section{Determinant diagrams and irreducible areas}\label{section:determinant diagrams}

\subsection{Determinant diagrams} \label{subsection:determinant diagrams}

We prove here propositions \ref{prop:flow class iff no irreducible area}, \ref{prop:flow implies conjugate} and \ref{coro:determinant diagrams with walls}. We use the two following lemmas of \cite{FLR2013} in the proofs:
\begin{lemma}[Le Roux \cite{FLR2013}, lemma $1.8$] \label{lemma:generalized Alexander trick}
Let $\F$ be a finite family of pairwise disjoint topological lines in the plane. Let $h_0$ be an orientation preserving homeomorphism of the plane. Let $H_{h_0}$ be the space of orientation preserving homeomorphisms that coincide with $h_0$ on the union of the elements of $\F$. Then $H_{h_0}$ is arcwise connected.
\end{lemma}

\begin{lemma}[Le Roux \cite{FLR2013}, lemma $1.6$] \label{lemma:flow classes iff streamlines} The Brouwer mapping class $[h;\Or]$ is a fixed point free flow class if and only if it admits a family of pairwise disjoint proper geodesic homotopy streamlines whose union contains $\Or$.
\end{lemma}

\begin{prop*}\emph{\textbf{\ref{prop:flow class iff no irreducible area}}} A Brouwer mapping class $[h;\Or]$ is a flow class if and only if no connected component of the complement of the set of walls for $[h;\Or]$ is an irreducible area.
\end{prop*}

\begin{proof}
If $[f;\Or]$ is a flow class, then according to lemma \ref{lemma:flow classes iff streamlines} we can choose a family of pairwise disjoint proper geodesic homotopy streamlines whose union contains $\Or$. We find reducing lines in the neighborhood of each streamline, hence there is no irreducible area.

We now prove that if the set of walls $\mathcal{W}$ of a Brouwer mapping class $[h;\Or]$ is such that no component of the complement of $\mathcal W$ is irreducible, then it is a flow class. According to proposition \ref{prop: nice families disjoint from reducing lines}, there exists a nice family $(\alpha_i^\pm)_i$ for $[h;\Or]$ disjoint from the walls. According to theorem \ref{theorem: walls}, every connected component of the complement of $\mathcal{W}$ which contain orbits is a translation area. According to proposition \ref{prop:translation area}, every backward proper which is included in a translation area is also forward proper: it follows that every  $T(\alpha_i^-,h,\Or)$ is a proper streamline. Lemma \ref{lemma:flow classes iff streamlines} gives us the conclusion.
\end{proof}

\begin{prop*}\emph{\textbf{\ref{prop:flow implies conjugate}}}
If two flow classes have the same diagram, then they are conjugated.
\end{prop*}

\begin{proof}
Let $[f;\Or]$ and $[g;\Or']$ be two flow classes with the same diagram. According to lemma \ref{lemma:flow classes iff streamlines}, there exists a nice family $(\alpha_i^\pm)_i$ for $[f;\Or]$ and a nice family $(\beta_i^\pm)_i$ for $[g;\Or']$ such that for every $i$, $\alpha_i^-$ is isotopic to $\alpha_i^+$, and $\beta_i^-$ is isotopic to $\beta_i^+$. We set $\alpha_i:=\alpha_i^-=\alpha_i^+$ and $\beta_i:=\beta_i^-=\beta_i^+$.
Since $[f;\Or]$ and $[g;\Or']$ have the same diagram, $(\alpha_i^\pm)_i$ and $(\beta_i^\pm)_i$ have the same cyclic order at infinity (we permute the numbering of the orbits of $\Or'$ if necessary). Thus the Schoenflies theorem provides a homeomorphism of the plane which sent $T(\alpha_i,h,\Or)$ to $T(\beta_i,h,\Or)$ for every $i$. Lemma \ref{lemma:generalized Alexander trick} gives the conclusion. \end{proof}

\begin{prop*}\emph{\textbf{\ref{coro:determinant diagrams with walls}}}
A diagram with walls without crossing arrows is determinant if and only if the arrows of every family of arrows included in the same connected component of the complement of the walls are backward adjacent and forward adjacent.
\end{prop*}

\begin{proof} Let $D$ be a diagram with walls without crossing arrows. Suppose $D$ is determinant. Since $D$ is without crossing arrows, there exists a flow class $[f;\Or]$ whose associated diagram is $D$ (this is lemma $1.7$ of \cite{FLR2013}). Since $[f;\Or]$ is a flow class, every orbit of $\Or$ is included in a translation area, hence in a maximal translation area. In this situation, theorem \ref{theorem: walls} imply that every connected component of the complement of walls which contains orbit is a maximal translation area. The result follows.

If a diagram with walls $D$ is such that every family of arrows included in the same connected component of the complement of the walls are backward adjacent and forward adjacent, then for every $[h;\Or]$ whose associated diagram is $D$, the complement of the set of walls in $\R^2$ has no irreducible areas (it has only translation areas and areas without orbits). According to proposition \ref{prop:flow class iff no irreducible area}, $[h;\Or]$ is a flow class. If $[h';\Or']$ is another Brouwer class whose associated diagram with walls is $D$, then $[h';\Or']$ is also a flow class. It follows from proposition \ref{prop:flow implies conjugate} that $[h;\Or]$ and $[h';\Or']$ are conjugated. Hence the diagram with walls $D$ is determinant.
\end{proof}

\subsection{Combinatorics of irreducible areas}\label{subsection:combinatorics of irreducible areas}
We first prove a criterion for reducing lines and then use it to prove proposition \ref{prop:combinatoric of an irreducible area}.

\subsubsection{A criterion for reducing lines}

\begin{lemma} \label{lemma: criterion for reducing lines}
Let $[h;\Or]$ be a Brouwer mapping class and let $\al$ be a nice family for $[h;\Or]$. If $\Delta$ is a topological line of $\R^2 - \Or$ such that:
\begin{enumerate}
\item $\Delta$ is a topological line;
\item Both components of $\R^2 - \Delta$ contain points of $\Or$;
\item For every $i$, for every $k\in \Z$, $\Delta$ is homotopically disjoint from $h^k(\alpha_i^-)$ relatively to $\Or$. \label{H3}
\end{enumerate}
Then $\Delta$ is a reducing line for $[h;\Or]$.
\end{lemma}

\begin{proof}
We need to show that $\Delta$ and $h(\Delta)$ are isotopic relatively to $\Or$. Choose a hyperbolic metric on $\R^2 - \Or$. Taking its image by an isotopy relatively to $\Or$ if necessary, we can assume that $\Delta$ is geodesic. We denote by $f$ a representative of $[h;\Or]$ mapping $\Delta$ on $h(\Delta)_\#$.
Such an $f$ exists, again according to the straightening principle \ref{lemma: straightening principle}. Hence $\Delta$ and $f(\Delta)$ are geodesic. 
We need to show that $\Delta=f(\Delta)$. Suppose that $\Delta \neq f(\Delta)$. We know that this two streamlines are in minimal intersection position (because they are geodesic), and we study separately the three possible cases: either $\Delta$ and $f(\Delta)$ have several intersection points, either they have only one intersection point, or they do not intersect. Those three cases lead us to contradictions.

\begin{figure}[h]
\labellist
\small\hair 2pt
\pinlabel $\Delta$ at 437 135
\pinlabel $f(\Delta)$ at 402 204
\pinlabel $\gamma_1$ at 180 186
\pinlabel $\delta_k$ at 175 67
\pinlabel $\gamma_k$ at 167 101
\pinlabel $a_1$ at 66 54
\pinlabel $b_1$ at 268 93
\endlabellist
\centering
\includegraphics[scale=0.7]{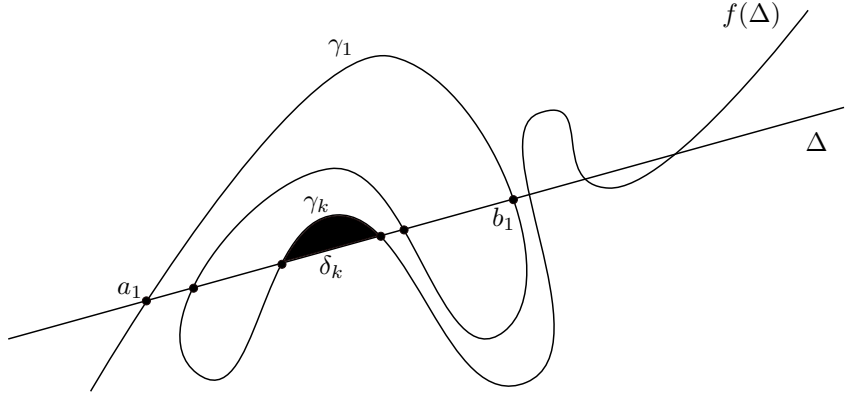}
\vspace{0.6cm}
\caption{Example where $\Delta$ and $f(\Delta)$ have several intersection points.}
\label{figu:delta}
\end{figure}

\noindent \textbf{If $\Delta$ and $f(\Delta)$ have several intersection points.} We consider a subsegment $\gamma_1$ of $f(\Delta)$, whose endpoints are intersection points between $\Delta$ and $f(\Delta)$, denoted by $a_1$ and $b_1$, and such that the open segment $\mathring \gamma_1$ is disjoint from $\Delta$ (see figure \ref{figu:delta}). Denote by $\delta_1$ the subsegment of $\Delta$ between $a_1$ and $b_1$. Since $\delta_1$ is compact, it contains finitely many intersection points between $\Delta$ and $f(\Delta)$. Denote by $n$ the number of intersections between $\Delta$ and $f(\Delta)$. Let us show that we can assume that $n=0$. If $n>0$, choose $a_2$ an intersection point between $\mathring \delta_1$ and $f(\Delta)$.
Consider the half line $f(\Delta)^+$ defined as the connected component of $f(\Delta) - a_2$ which has a subsegment with endpoint $a_2$ and which is included in the topological disk bounded by $\gamma_1 \cup \delta_1$. 
Because $f(\Delta)^+$ is proper, it goes out the topological disk bounded by $\gamma_1 \cup \delta_1$. Hence it intersects again $\delta_1$, because $f(\Delta)$ is without self intersection. Denote by $b_2$ the first intersection point between $f(\Delta)^+$ and $\delta_1$. The subsegments $\delta_2$ and $\gamma_2$ of $\Delta$ and $f(\Delta)$ with endpoints $a_2$ and $b_2$ have the same properties than $\delta_1$ and $\gamma_1$, but the number of intersection points between $\Delta$ and $f(\Delta)$ on $\delta_2$ is less than the one on $\delta_1$. Hence by recurrence there exits $k$ and two subsegments $\delta_k \subset \Delta$ and $\gamma_k \subset f(\Delta)$ with endpoints $a_k$ and $b_k$, and such that $\delta_k \cap f(\Delta) = \gamma_k \cap \Delta = a_k \cup b_k$. 
 
 Denote by $D$ the topological disk bounded by $\delta_k \cup \gamma_k$. We claim that $D$ does not intersect $\Or$: if it contains a point of an orbit $\Or_i$, denote by $x_i$ this point. Let $n \in \Z$ be such that $x_i$ is an endpoint of $h^n(\alpha_i^-)$. The family $(h^k(\alpha_i^-))_{k\leq n}$ is proper, because $\alpha_i^-$ is backward proper. Applying again the straightening principle \ref{lemma: straightening principle}, we find a homeomorphism $g$ isotopic to $h$, which maps $\Delta$ on $h(\Delta)_\#=f(\Delta)$ and $h^k(\alpha_i^-)$ on $(h^{k+1}(\alpha_i^-))_\#$ for every $k\leq n$. Moreover, the family $(h^k(\alpha_i^-))_{k\leq n}$ is homotopically disjoint from $\Delta$, according to the third hypothesis of the lemma. Hence $g^k(\alpha_i)$ is disjoint from $\Delta$ and $f(\Delta)$ for every $k \leq n$. It follows that $\{h^k(x_i)\}_{k\leq n}$ is included in $D$. Since the orbits of a Brouwer homeomorphism are proper, this gives a contradiction: $D$ should be disjoint from $\Or$, hence it is a bigone, which is also not possible because $\Delta$ and $f(\Delta)$ are in minimal position.\\

\noindent \textbf{If $\Delta$ and $f(\Delta)$ have exactly one intersection point.} The set $\Delta \cup f(\Delta)$ splits $\R^2$ into $4$ connected components. Each of those connected components contains at least one orbit of $\Or$: if not, we can find an isotopy relatively to $\Or$ which eliminate the intersection point, hence $\Delta$ and $f(\Delta)$ are not in minimal position. Choose an orientation on $\Delta$ and consider the induced orientation by $f$ on $f(\Delta)$. Then there exists at least one orbit of $\Or$ which is on the left of $\Delta$ and on the right of $f(\Delta)$. We claim that this is not possible: 
\begin{itemize}
\item $\Delta$ splits $\R^2$ into two topological half plane, denoted by $\Pc$ and $\Qc$.
\item $\Delta$ splits the orbits of $\Or$ into two families: the family $P$ is the orbits included in $\Pc$ and the family $Q$ is the orbits included in $\Qc$. 
\item For every orbit $\Or_i$ of $\Or$, we have $f(\Or_i)=\Or_i$.
\end{itemize}
Then $f(\Delta)$ splits $\R^2$ into $f(\Pc)$ and $f(\Qc)$, hence the orbits into $f(P)=P$ and $f(Q)=Q$.\\

\noindent \textbf{If $\Delta$ and $f(\Delta)$ have no intersection point.} The set $\Delta \cup f(\Delta)$ splits $\R^2$ into three connected components. One of them is between $\Delta$ and $f(\Delta)$. This component contains at least one orbit, because $\Delta$ and $f(\Delta)$ are not isotopic. Now the same argument than in the previous case leads us to a contradiction: choose an orientation on $\Delta$ and consider the induced orientation by $f$ on $f(\Delta)$. There exists at least one orbit of $\Or$ which is on the left of $\Delta$ and on the right of $f(\Delta)$. This is not possible. \end{proof}

\begin{corollary} \label{corollary: reducing lines included in stable area are reducing lines relatively to every orbits} Let $[h;\Or]$ be a Brouwer mapping class.
Let $Z$ be a stable area for $[h;\Or]$ which contains at least two orbits of $\Or$. Choose a complete family of adjacency areas for $[h;\Or \cap Z]$. Let $\Delta$ be a reducing line for $[h;\Or\cap Z]$ which is included in $Z$ and disjoint from every chosen backward adjacency area. Then $\Delta$ is a reducing line for $[h;\Or]$.
\end{corollary}

\begin{proof}
Choose a complete hyperbolic metric on $\R^2 - \Or$. We suppose that the boundary components of $Z$ and $\Delta$ are geodesic. The straightening principle \ref{lemma: straightening principle} gives us $h' \in [h;\Or]$ which preserves $Z$. Let $\al$ be a nice family disjoint from the boundary component of $Z$, such that every $\alpha_i^-$ of $Z$ is included in an adjacency area of the chosen family.

Since $\Delta$ is disjoint from every chosen backward adjacency areas, it is disjoint from $\alpha_i^-$ for every $i$ such that $\alpha_i^-$ is in $Z$. Hence for every $k\in \Z$, $\Delta$ is homotopically disjoint from $h'^k(\alpha_i^-)$ relatively to $\Or \cap Z$ (because $\Delta$ is isotopic to its image by $h'$ relatively to $\Or \cap Z$). Since $h'^k(\alpha_i^-)$ is compact and included in $Z$, which is preserved by $h'$, we get that $\Delta$ and $h'^k(\alpha_i^-)$ are homotopically disjoint relatively to $\Or$ (and not only relatively to $\Or \cap Z$. 

It follows from lemma \ref{lemma: criterion for reducing lines} that $\Delta$ is a reducing line for $[h;\Or]$.
\end{proof}

\subsubsection{Proof of proposition \ref{prop:combinatoric of an irreducible area}}

\begin{prop*}\emph{\textbf{\ref{prop:combinatoric of an irreducible area}} (Combinatorics of irreducible areas).}
Let $[h;\Or]$ be a Brouwer mapping class and let $Z$ be an irreducible area for $[h;\Or]$. Then:
\begin{enumerate}
\item The orbits of $Z$ are not all backward adjacent, neither all forward adjacent for $[h;\Or]$;
\item $Z$ has at least two boundary components;
\item The orbits of $[h;\Or \cap Z]$ do not alternate.
\end{enumerate}
\end{prop*}

\noindent \textbf{Proof of $(1)$.} 
Assume that $Z$ intersects only one forward adjacency area. According to lemma \ref{lemma:fitted family}, every $\alpha_i^- \in Z$ is forward proper. Hence $Z$ is not irreducible.\\

\noindent \textbf{Proof of $(2)$.}
Let $[h;\Or]$ be a Brouwer mapping class. Let $Z$ be a stable area for $[h;\Or]$ which has only one boundary component and at least two orbits. Denote by $L$ this boundary component. Suppose that $Z$ is not a translation area. We will find a reducing line for $[h;\Or]$, included in $Z$ and non isotopic to $L$. Let $\al$ be a nice family for $[h;\Or]$ disjoint from $L$. There is a subfamily of $\al$ which is a nice family for $[h;\Or \cap Z]$. Denote by $\bet$ this subfamily. We consider the cyclic order of $\bet$, and look where is $L$ in this cyclic order: the position of $L$ in the cyclic order is the position of $L$ in the plane relatively to the half homotopy streamlines generated by the $\beta_i^\pm$'s. There are two different cases:
\begin{enumerate}
\item [(a).] If $L$ is between two backward proper arcs or between two forward proper arcs in the cyclic order of $\bet$;
\item [(b).] If $L$ is between one backward proper arc and one forward proper arc in the cyclic order of $\bet$.
\end{enumerate}

\noindent \textbf{Case (a).}
If $L$ is between two backward proper arcs or between two forward proper arcs in the cyclic order of $\bet$, we claim that there exists an adjacency area for $[h;\Or\cap Z]$ which contains $L$. Indeed, assume that $L$ is between two backward proper arcs (the same proof holds with two forward proper arcs, replacing $h$ by $h^{-1}$ when it is necessary). Denote by $\beta_i^-$ and $\beta_j^-$ this two backward proper arcs, and suppose their endpoints are respectively $x_i,h(x_i)$ and $x_j,h(x_j)$. Then there exists an arc $\gamma$ disjoint from $L$, whose endpoints are $h(x_i)$ and $h(x_j)$ and such that one connected component of the complement of $T^-(\beta_i^-,h,\Or) \cup \gamma \cup T^-(\beta_j^-,h,\Or)$ does not intersect $\Or \cap Z$. Now Handel's theorem \ref{theo:reducing line disjoint from backward adjacency areas} implies that there exists a reducing line $\Delta$ for $[h;\Or\cap Z]$ which is disjoint from every backward adjacency area. Hence $\Delta$ is included in $Z$, and according to corollary \ref{corollary: reducing lines included in stable area are reducing lines relatively to every orbits}, it is a reducing line for $[h;\Or]$.\\

\noindent \textbf{Case (b).} Assume $L$ is between one backward proper arc and one forward proper arc in the cyclic order of $\bet$. Denote by $\beta_i^-$ and $\beta_j^+$ this two arcs. Following the construction \ref{construction of adjacency areas}, we get a complete family of adjacency areas for $[h;\Or \cap Z]$ disjoint from $L$. Now let $\Delta$ be a reducing line for $[h;\Or \cap Z]$ given by lemma \ref{lemma:existence of a reducing line almost included in forward adjacency areas}, i.e. such that the intersection between $\Delta$ and the complement of the forward adjacency areas of $[h;\Or\cap Z]$ is bounded. It follows that $\Delta$ intersects $L$ at most in a finite set (because $L$ is disjoint from the adjacency areas). Isotopying $\Delta$ if this set is not empty, we can suppose that $\Delta \cap L = \emptyset$. Since $\Delta$ is also disjoint from every backward adjacency areas of $[h;\Or \cap Z]$, according to corollary \ref{corollary: reducing lines included in stable area are reducing lines relatively to every orbits}, it is a reducing line for $[h;\Or]$.\\

\noindent \textbf{Proof of $(3)$.} This is a consequence of lemma $3.6$ of \cite{FLR2013}: every family of alternating orbits satisfies the uniqueness of homotopy translation arcs. Suppose that the orbits of $[h;\Or \cap Z]$ are alternate. According to this lemma and to proposition \ref{prop: nice families disjoint from reducing lines}, if $\al$ is a nice family for $[h;\Or]$ disjoint from the boundary components of $Z$, then for every $i$ such that $\alpha_i^-$ is in $Z$, $\alpha_i^-$ and $\alpha_i^+$ are isotopic relatively to $\Or \cap Z$. Since they are included in $Z$, there are also isotopic relatively to $\Or$. Hence for such an $i$, $T(\alpha_i^-,h,\Or)$ is a proper streamline. At least one of the boundary components of a tubular neighborhood of this streamline is a reducing line included in $Z$. Hence $Z$ is not irreducible. \qed

\subsection{Corollaries of proposition \ref{prop:combinatoric of an irreducible area}} \label{subsection:corollaries of irreducible areas}

\subsubsection{Proof of corollary \ref{corollary:r'=2 implies flow class}}

\begin{corollary*} \emph{\textbf{\ref{corollary:r'=2 implies flow class}}}
Let $[h;\Or]$ be a Brouwer mapping class relatively to $r$ orbits. Denote by $2r'$ the number of adjacency subfamilies of $[h;\Or]$. If $r'= 1,2$ or $r$, then $[h;\Or]$ is a flow class.
\end{corollary*}

\begin{proof}
If $[h;\Or]$ is not a flow, then it has an irreducible area for some reducing set (according to proposition \ref{prop:flow class iff no irreducible area}). This irreducible area has at least two boundary components (according to proposition \ref{prop:combinatoric of an irreducible area}), which are reducing lines. Denote by $\Delta_1$ and $\Delta_2$ those two boundary components. The complement of $\Delta_1 \cup \Delta_2$ has three components, denoted by $Z_1$, $Z_2$ and $Z$. Assume that $Z$ is the area in the middle, which contains the irreducible area. Choose a nice family $\al$ for $[h;\Or]$ which is disjoint from $\Delta_1$ and $\Delta_2$ (use proposition \ref{prop: nice families disjoint from reducing lines}). Since $Z$ contains an irreducible area, according to proposition \ref{prop:combinatoric of an irreducible area} it intersects at least two different backward adjacency areas of $[h;\Or]$ and at least two different forward adjacency areas of $[h;\Or]$, and the orbits of $[h;\Or \cap Z]$ do not alternate. Hence the situation is the one of figure \ref{figu:irreducible-fig}: there exists a subfamily $(\alpha_{i_1}^-,\alpha_{i_2}^+,\alpha_{i_3}^-,\alpha_{i_4}^+)$ of $\al$ containing only arcs included in $Z$ and such that the cyclic order of this subfamily is $(\alpha_{i_1}^-,\alpha_{i_2}^+,\alpha_{i_4}^+,\alpha_{i_3}^-)$, with $\Delta_1$ between $\alpha_{i_2}^+$ and $\alpha_{i_4}^+$ and $\Delta_2$ between $\alpha_{i_3}^-$ and $\alpha_{i_1}^-$. Since $Z_1$ and $Z_2$ contain at least one orbit (because $\Delta_1$ and $\Delta_2$ are reducing lines), there exist a backward proper arc of $\al$ in $Z_1$, denoted by $\alpha_{i_5}^-$, and a forward proper arc of $\al$ in $Z_2$, denoted by $\alpha_{i_6}^+$. It follows that $\al$ has a subfamily of arcs whose cyclic order at infinity is $(\alpha_{i_1}^-,\alpha_{i_2}^+,\alpha_{i_5}^-,\alpha_{i_4}^+,\alpha_{i_3}^-,\alpha_{i_6}^+)$. Hence $r'\geq 3$. It was shown in \cite{FLR2013}, lemma $6.6$, that $r'<r$.\end{proof}

\begin{figure}[h]
\begin{center}
\labellist
\small\hair 2pt
\pinlabel $-$ at 23 85
\pinlabel $3$ at 22 95

\pinlabel $+$ at 55 85
\pinlabel $4$ at 56 95

\pinlabel $+$ at -25 40
\pinlabel $-$ at -10 40
\pinlabel $6$ at -17 53

\pinlabel $+$ at 105 40
\pinlabel $-$ at 90 40
\pinlabel $5$ at 98 53

\pinlabel $-$ at 24 -5
\pinlabel $1$ at 23 -14

\pinlabel $+$ at 57 -5
\pinlabel $2$ at 58 -14 

\endlabellist
\centering
\vspace{0.2cm}
\includegraphics[scale=1]{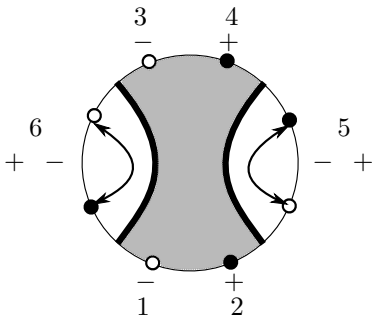}
\vspace{0.5cm}
\caption{Combinatorics of an irreducible area}
\label{figu:irreducible-fig}
\end{center}
\end{figure}

\subsubsection{Proof of corollary \ref{cor:deux orbites bordantes}}

\begin{corollary}\label{cor:deux orbites bordantes}
Let $[h;\Or]$ be a Brouwer mapping class relatively to $r$ orbits.
\begin{enumerate}
\item If $r\geq 3$, there exist at least two disjoint and non isotopic reducing lines for $[h;\Or]$;
\item If $r \geq 2$, there exist at least two translation areas for $[h;\Or]$ which have exactly one boundary component;
\item There exists a nice family $\al$ and $j \neq k$ such that:
\begin{itemize}
\item Relatively to $\Or$, $\alpha_j^-$ is isotopic to $\alpha_j^+$ and $\alpha_k^-$ is isotopic to $\alpha_k^+$;
\item In the cyclic order, $\alpha_j^-$ and $\alpha_j^+$ (respectively $\alpha_k^-$ and $\alpha_k^+$) are neighbors.
\end{itemize}
\end{enumerate}
\end{corollary}

\noindent \textit{Proof of $(1)$.} Let $r$ be greater than $2$. Theorem \ref{theo:existence of reducing line} gives us a first reducing line. This line splits the plane into two stable areas. One of them contains at least two orbits. According to $(2)$ of proposition \ref{prop:combinatoric of an irreducible area}, every stable area with one boundary component and which contains at least two orbits contains at least one reducing line non isotopic to the boundary component: this gives a second reducing line.\qed \\

\noindent \textit{Proof of $(2)$ and $(3)$.} If $r=2$ then $[h;\Or]$ any reducing line split the plane into two translation areas. If $r\geq 3$ we find two reducing lines as done in the proof of $(1)$. In the complement of these two reducing lines we have in particular two adjacency areas with one boundary component. In each of these areas, finding again a reducing line in the area as done in $(1)$ if necessary, inductively we find a stable area with one orbit and one boundary components: this gives $(2)$ and $(3)$.\qed \\

\section{Deflectors} \label{section: deflectors}

This section is independent of sections \ref{section:Adjacency areas, diagrams and special nice families}, \ref{section: walls} and \ref{section:determinant diagrams}. The main result is proposition \ref{prop: deflector}, that we will need to prove theorem \ref{theo:total invariant}.

\begin{prop}\label{prop: deflector}
Let $\tau:\R^2 \rightarrow \R^2$ be the horizontal translation $(x,y)\mapsto (x+1,y)$. Let $n \in \N$. Let $(\alpha_i)_{1 \leq i \leq n}$ and $(\beta_i)_{1 \leq i \leq n}$ be two families of translation arcs for $\tau$ such that:
\begin{itemize}
\item For every $i$, $\alpha_i$ and $\beta_i$ join $(0,i)$ to $(1,i)$.
\item The $\alpha_i$'s (respectively the $\beta_i$'s) are mutually disjoint.
\end{itemize}
Then there exists a homeomorphism $\mu$ of $\R^2$ with a compact support and such that:
\begin{enumerate}
\item $\mu(\Z \times \{1,...,n\})=\Z \times \{1,...,n\}$;
\item For every sufficiently large $k \in \N$, for every $i$, $(\mu \tau)^k(\alpha_i)$ is isotopic relatively to $\Z \times \{1,...,n\}$ to $\tau^k(\beta_i)$;
\item $\mu \tau$ is a Brouwer homeomorphism; More precisely, $\mu$ is a finite product of $\tau$-free half twists disjointly supported.
\end{enumerate}
\end{prop}

\begin{definition}
Such a homeomorphism is called a \emph{deflector} associated to $(\alpha_i,\beta_i)_{1 \leq i \leq n}$.
\end{definition}

Let $\C_n$ be the open vertical cylinder with $n$ marked points at distinct heights. 
Recall that $MCG(\C_n)$ is defined as the quotient of the group of homeomorphisms of the cylinder, fixing each boundary puncture and fixing the set of marked points (not necessary point wise), by its connected component of the identity (for the compact-open topology). In particular, it is the subgroup of the braid group of the $(n+2)$-punctured sphere $B_{n+1}(\Sph^2)$ which fixes two punctures.

We use the following lemma in the proof of proposition \ref{prop: deflector}.
\begin{lemma}\label{lemma: braids for deflector}
Let $(\gamma_i)_{1\leq i \leq n}$ be a family of mutually disjoint simple closed curves on $\C_n$, such that each curve contains exactly one marked point, and such that each curve is isotopic in the cylinder without marked point to the separating circle. Let $(\gamma'_i)_{1\leq i \leq n}$ be the family of disjoint horizontal circles on the cylinder, such that each $\gamma'_i$ contains a marked point.

Then there exists $\phi \in MCG(\C_n)$ such that:
\begin{itemize}
\item For every $i$, there exists $j$ such that $\phi(\gamma_i)$ is isotopic to $\gamma'_j$ in $\C_n$;
\item $\phi$ is a finite product of half-twists.
\end{itemize}
\end{lemma}

\noindent \textit{Proof of lemma \ref{lemma: braids for deflector}.} We denote again by $\gamma_i$, $\gamma'_i$ the isotopy classes of $\gamma_i$, $\gamma'_i$ when there is no confusion. Let $\varphi \in MCG(\C_n)$ such that $(\varphi(\gamma_i))_i=(\gamma'_i)_i$. Then if $T$ is a product of horizontal Dehn twists (around one of the $\gamma'_i$'s or around a boundary component), $T\varphi$ coincides with $\varphi$ on the $\gamma_i$'s.

\noindent \textbf{The group $MCG(\C_n)$ as a quotient of a subgroup of the braid group.} Denote by $B_n$ the usual braid group, i.e. the mapping class group of the disk with $n+1$ marked point, denoted by $x_1,...,x_{n+1}$. Let $G$ be the subgroup of $B_n$ which fixes $x_{n+1}$. Thus $MCG(\C_n)$ is the quotient of $G$ by its center (which is generated by the Dehn twist around the boundary component of the disk). 

We will need the three followings to prove the lemma.\\

\noindent \textbf{(1) Linking number of a pure braid.} Denote by $P_n$ the subgroup of $B_n$ composed by the pure braids, i.e. the braids which fix point wise the marked point $x_1,...,x_n$. Let $\rho \in P_n$ be a pure braid.
In the geometric representation of $\rho$, for every $i\leq n$, the strand from $x_i$ can turn around the strand from $x_{n+1}$ clockwise or counterclockwise. We count $+1$ each time it turns around clockwise, and $-1$ each time it turns around counterclockwise. We call \emph{linking number of $x_i$ around $x_{n+1}$} and denote by $\epsilon_i(\rho)$ the sum obtained when we look over the strand from $x_i$. We define the \emph{total linking number of $\rho$} as the sum $\epsilon(\rho):=\Sigma_{i=1}^n \epsilon_i(\rho)$. Note that $\epsilon$ is a morphism from $P_n$ to $\Z$. \\

\noindent \textbf{(2) Special form of a pure braid (see for example \cite{Kassel-Turaev}).}
The braid $A_{i,j} \in B_n$ is usually defined as in figure \ref{figu:A_ij}. More precisely, if we denote by $\sigma_k$ the usual half twists which generate $B_n$, we have:
$$A_{i,j}=\sigma_{i-1} \sigma_{i-2}...\sigma_{j+1}\sigma_j^2 \sigma_{j+1}^{-1}...\sigma_{i-2}^{-1}\sigma_{i-1}
^{-1}.$$

\vspace{0.3cm}
\begin{figure}[h]
\labellist
\small\hair 2pt
\pinlabel $i$ at 41 -10
\pinlabel $j$ at 104 -10
\endlabellist
\centering
\includegraphics[scale=0.5]{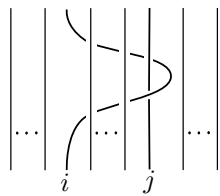}
\vspace{0.5cm}
\caption{The braid $A_{i,j}$.}
\label{figu:A_ij}
\end{figure}

\noindent \textbf{Claim A. (see for example \cite{Kassel-Turaev})} Let $\rho \in P_n$ be a pure braid. Then $\rho$ can be written as: $\rho=\beta_2...\beta_n\beta_{n+1}$, where every $\beta_k$ is in the free group generated by $A_{1,k},...,A_{k-1,k}$.\\

\noindent \textbf{(3) Claim B.} Let $\rho$ be in the group generated by $A_{1,n+1},...,A_{n,n+1}$. Assume that the linking number of $\rho$ is trivial. Then $\rho$ can be written as a finite product of half twists supported in topological disks disjoint from $x_{n+1}$.\\

\noindent \textit{Proof of claim B.} Such an element can be written as:
$$\rho = A_{i_1,n+1}^{k_1}A_{i_2,n+1}^{k_2}...A_{i_l,n+1}^{k_l}A_{i_{l+1},n+1}^{-(k_1+k_2+...+k_l)},$$
where the $k_j$'s are in $\Z$, the $i_j$'s are integers between $1$ and $n$, and $l$ is a non negative integer. Note that $\epsilon(A_{k,n+1})=1$ for every $k$. Since $\epsilon$ is a morphism, $\epsilon(\rho)=0$ implies that the sum of the powers is trivial.

\begin{figure}[h]
\labellist
\small\hair 2pt
\pinlabel $1$ at 20 80
\pinlabel $i$ at 60 80
\pinlabel $j$ at 100 80
\pinlabel $n+1$ at 145 80

\endlabellist
\centering
\vspace{0.3cm}
\includegraphics[scale=0.6]{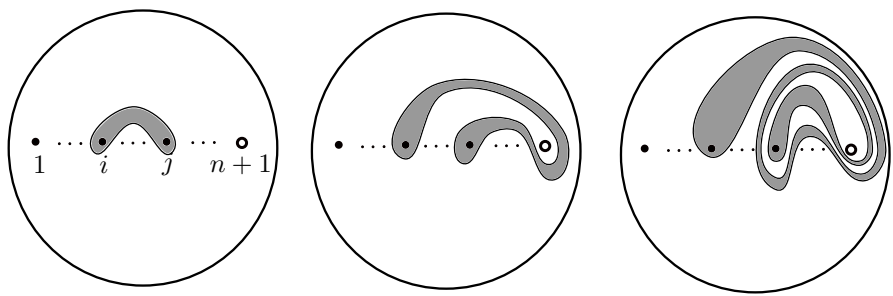}
\vspace{-0.2cm}
\caption{Topological disks $D_{i,j}$, $\sigma_{j,n+1}^{2}(D_{i,j})$ and $\sigma_{j,n+1}^{4}(D_{i,j})$.}
\vspace{0.3cm}
\label{figu:twist}
\end{figure}

\begin{figure}
\labellist
\small\hair 2pt
\pinlabel $i$ at 2 -15
\pinlabel $j$ at 45 -18
\pinlabel $n+1$ at 105 -19
\endlabellist
\centering
\begin{center}
\includegraphics[scale=0.3]{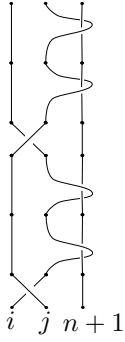}
\vspace{0.3cm}
\caption{$A_{i,n+1}^2A_{j,n+1}^{-2}=\sigma_{j,n+1}^4\sigma_{i,j}\sigma_{j,n+1}^{-4}\sigma_{i,j}^{-1}$.}
\label{figu:chemin}
\end{center}
\end{figure}

Denote by $\sigma_{i,j}$ the half twist supported in a topological disk $D_{i,j}$ containing $x_i$ and $x_j$ and lying above every other marked points, as the left item of figure \ref{figu:twist}. Note that:
\begin{itemize}
\item The conjugate of $\sigma_{i,j}$ by $\sigma_{j,n+1}^{2k}$ is the half twist supported in $\sigma_{j,n+1}^{2k}(D_{i,j})$ (see figure \ref{figu:twist});
\item $A_{i,n+1}^k A_{j,n+1}^{-k}= \sigma_{j,n+1}^{2k} \sigma_{i,j} \sigma_{j,n+1}^{-2k} \sigma_{i,j}^{-1}$ (see figure \ref{figu:chemin}).
\end{itemize}  
Thus we know how to realize every $A_{i,n+1}^k A_{j,n+1}^{-k}$ as a product of $4$ half twists supported in $D_{i,j}$ and $\sigma^{2k}_{j,n+1}(D_{i,j})$, hence as a product of half twists supported in topological disks disjoint from $x_{n+1}$. Now remark that $\rho$ can be written as:
$$\rho = A_{i_1,n+1}^{k_1}A_{i_2,n+1}^{-k_1} A_{i_2,n+1}^{k_1+k_2} A_{i_3,n+1}^{-(k_1+k_2)} ...A_{i_l,n+1}^{k_1+...+k_l}A_{i_{l+1},n+1}^{-(k_1+k_2+...+k_l)}.$$ 
Hence we know how $\rho$ as a finite product of half twists supported in topological disks disjoint from $x_{n+1}$. \qed \\

\noindent \textbf{Back to the proof of lemma \ref{lemma: braids for deflector}.}
Choose a topological disk $K$ of $\C$ which contains every $x_i$ for $i \leq n$. Every element of $B_n$ supported in $K$ is in $G$. Choose a lift $\hat \varphi$ of $\varphi$ in $G$. Compose $\hat \varphi$ with an element $\sigma \in G$ supported in $K$ and such that $\sigma \hat \varphi$ is a pure braid.

(1) Note that the linking number of $\sigma \hat \varphi$ does not depend on $\sigma$, because its supported in $K$. Up to composing $\varphi$ by a finite product of horizontal Dehn twists, we can assume that $\epsilon(\sigma \hat \varphi)=0$.

(2) Since $\sigma  \hat \varphi$ is a pure braid, according to claim A, we can write it as:
$$\sigma \hat \varphi = \beta_2...\beta_n\beta_{n+1},$$
where every $\beta_k$ is in the free group generated by $A_{1,k},...,A_{k-1,k}$.

For every $i \leq k \leq n$, $A_{i,k}$ is supported in $K$, hence $\beta_n...\beta_2$ is supported $K$, hence $\epsilon(\beta_{k})=0$ for every $k\leq n$. Since $\epsilon$ is a morphism, we have also $\epsilon(\beta_{n+1})=0$. 

(3) Because the braid group $B_{n-1}$ is generated by the usual half twists (see for example \cite{Kassel-Turaev}), we know how to write any element supported in $K$ as a product of half twists supported in topological disks disjoint from $x_{n+1}$. Now applying claim B to $\beta_{n+1}$, we also know how to write $\beta_{n+1}$ as a product of half twists supported in topological disks disjoint from $x_{n+1}$. Thus we know how to write $\sigma \hat \varphi$, hence $\hat \varphi$, as a finite product of half twists supported in topological disks disjoint from $x_{n+1}$. This product defines an element of $MCG(\C_n)$ which satisfies the property of the lemma. \qed

\vspace{5mm}

\noindent \textit{Proof of proposition \ref{prop: deflector}.}
Denote by $\C$ the vertical cylinder (quotient of $\R^2$ by $\tau$) and by $\C_n$ the cylinder $\C$ with $n$ marked points (quotient of $\R^2 - \Z \times \{1,...,n\}$ by $\tau$).
Let $\pi$ be the quotient map. For every $i$, we denote by $\hat \alpha_i$, respectively $\hat \beta_i$, the isotopy class of $\pi(\alpha_i)$, respectively $\pi(\beta_i)$, in $\C_n$. There exists $\psi \in MCG(\C_n)$ such that $\psi(\hat \alpha_i)$ is isotopic in $\C_n$ to the horizontal circle containing the marked point $x_i=\pi(\Z \times \{ i\})$. There exists $\chi \in MCG(\C_n)$ such that $\chi(\hat \beta_i)$ is isotopic in $\C_n$ to the horizontal circle containing the marked point $x_i=\pi(\Z \times \{ i\})$. We set $\phi:=\chi^{-1} \psi$. Hence $\phi(\hat \alpha_i)$  is isotopic in $\C_n$ to $\hat \beta_i$. According to lemma \ref{lemma: braids for deflector}, we can assume that $\phi:=\nu_1...\nu_k$ is a finite product of half twists supported in topological disks of $\C$. We want to use $\phi$ to construct the desired homeomorphism $\mu$ (figure \ref{figu:cylindre} gives the main idea of the proof).

\begin{figure}[h!]
\labellist\small\hair 2pt
\pinlabel $\phi$ at 256 61
\pinlabel {Area with the $\tilde \nu_i$'s} at 256 180
\endlabellist
\centering
\vspace{0.4cm}
\includegraphics[scale=0.6]{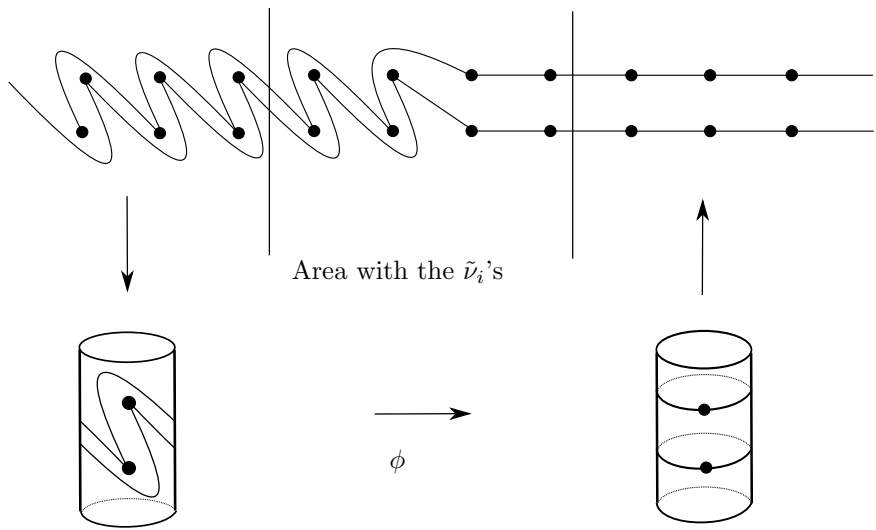}
\caption{Half streamlines $T^+(\alpha_i,\mu \tau,\Z\times \{1,...,n\})$.}
\label{figu:cylindre}
\end{figure}

\noindent \textbf{Local lift of a half twist.} We first construct a homeomorphism of $\R^2$ which lift the action of \emph{one} half twist of $\C_n$ to the lift of the curves. Let $\nu$ be a half twist of $\C_n$. Choose a lift $D_\nu$ of the support of $\nu$ in $\R^2$. Denote by $\tilde \nu$ the homeomorphism of $\R^2$ such that $\pi \tilde \nu_{|D_\nu} := \nu \pi_{|D_\nu}$, and such that $\tilde \nu$ coincides with $Id$ outside $D_\nu$. Note that $\Z\times \{1,...,k\}$ is preserved by $\tilde \nu \tau$. We say that $\tilde \nu$ is a \emph{local lift of $\nu$}.

\noindent \textbf{Choice of disks.} To lift the action of $\phi:=\nu_1...\nu_k$, we choose a local lift $\tilde \nu_k$ of $\nu_k$ supported in a disk $D_k$ on the right of the $\alpha_i$'s and for every $1 \leq j < k$, we choose a local lift $\tilde \nu_j$ of $\nu_j$ supported in a disk $D_j$ on the right of $D_{j+1}$ (and disjoint from $D_{j+1}$).

\noindent \textbf{Conclusion.} Let $\tilde \mu$ be $\tilde \nu_1 ... \tilde \nu_k$. For every $x$ which is on the left of the $D_j$'s, for every $n$ sufficiently large, there exists $(n_j)_j \in \N^{k+1}$ such that:
$$(\tilde \mu \tau)^n=\tau^{n_{k+1}} \tilde \nu_1 ... \tilde \nu_{k-1} \tau^{n_2} \tilde \nu_k \tau^{n_1} (x).$$
Moreover, we have the followings:
$$\pi \tilde \nu = \nu \pi;$$
$$\pi \tau = \pi.$$
Hence we get: $$\pi (\tilde \mu \tau)^n \alpha_i=\pi \beta_i. $$ \qed

\section{Classification relatively to $4$ orbits} \label{section: 4 orbits}

\subsection{Identification of the determinant diagrams}
 Here we want to identify which diagrams are determinant diagrams. In the next section, we will study the diagrams which are not determinant. Note that all the diagrams with four orbits are represented in the appendix \ref{appendix}.

For every Brouwer mapping class relatively to $4$ orbits, we denote by $2r'$ the number of sub-families of adjacency.

\begin{prop} \label{prop: irreducible areas for r=4}
A diagram for a Brouwer mapping class relatively to $4$ orbits is non determinant if and only if it is one of the seven of figure \ref{figu:non_determinant_diag_for_4_orbits}.
\end{prop}

\begin{figure}[h]
\begin{center}
\includegraphics[scale=0.7]{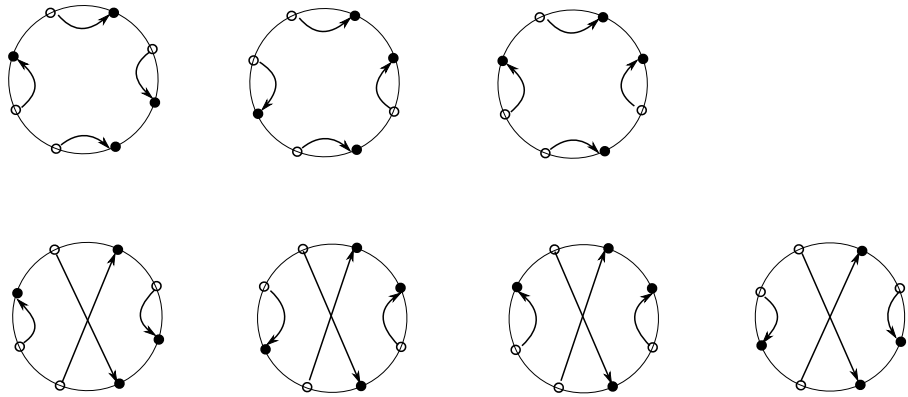}
\vspace{0.2cm}
\caption{Non determinant diagrams for $r=4$}
\label{figu:non_determinant_diag_for_4_orbits}
\end{center}
\end{figure}

\begin{proof}
According to proposition \ref{prop:combinatoric of an irreducible area}, every irreducible area for Brouwer mapping classes relatively to $4$ orbits is as in figure \ref{figu:zones-irr}. \end{proof}

\begin{figure}[h] 
\begin{center}
\includegraphics[scale=1]{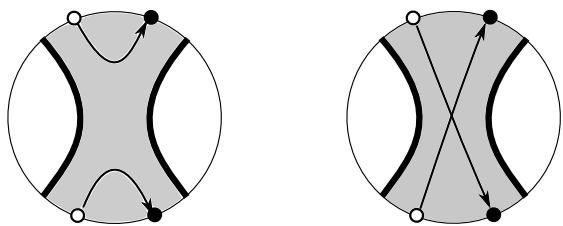}
\vspace{0.2cm}
\caption{Possible irreducible areas for $r=4$}
\label{figu:zones-irr}
\end{center}
\end{figure}

\subsection{Study of the non determinant diagrams}\label{subsection:study of non determinant diagrams}

\subsubsection{Brouwer mapping classes which realize non determinant diagrams}

If $h$ is a homeomorphism of the plane, recall that an $h$-free disk is a topological disk $D$ which is disjoint from every $h^n(D)$, with $n \neq 0$. If $[h;\Or]$ is a Brouwer mapping class and if $D$ is an $h$-free disk containing exactly two points of $\Or$, then we call \emph{free half twist} any half twist supported in $D$ and permuting the two points of $D \cap \Or$.

\begin{rem} \label{rmk:realization of non determinant diagrams} Each non determinant diagram can be realized by a Brouwer mapping class: 
\begin{itemize}
\item For each non determinant diagram without crossing, there exists a flow having this diagram relatively to some of its orbits (see lemma $1.7$ of\cite{FLR2013});
\item For each non determinant diagram with crossing, we can obtain it by composing a flow by a free half twist (as in example B of section \ref{section: First tools}).
\end{itemize}
\end{rem}

\subsubsection{Tangle of the irreducible area}\label{subsection: tangle}

Let $[h;\Or]$ be a Brouwer mapping class relatively to $4$ orbits which is not a flow class: the diagram with walls of $[h;\Or]$ is as in figure \ref{figu:irreducible are for diagram with walls}. To simplify the notations, suppose that the two orbits of the irreducible areas are $\Or_1$ and $\Or_2$, with $\alpha_1^\pm$ and $\alpha_2^\pm$ \ref{figu:irreducible are for diagram with walls}. To define the tangle, we will forget about $\Or_3$ and $\Or_4$ for a while.

\begin{figure}[h]
\begin{center}
\vspace{0.3cm}
\includegraphics[scale=1]{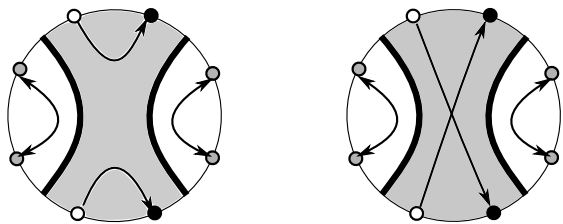}
\vspace{0.2cm}
\caption{Diagrams with an irreducible area for $r=4$.}
\label{figu:irreducible are for diagram with walls}
\end{center}
\end{figure}

Choose a nice family $\al$ for $[h;\Or]$. Choose a complete hyperbolic metric on $\R^2 -(\Or_1 \cup \Or_2)$.
According to corollary \ref{corollary:r'=2 implies flow class} and proposition \ref{prop:flow implies conjugate} (see also Handel \cite{Handel}), the diagram relatively to $2$ orbits is a total conjugacy invariant, hence $[h;\Or_1 \cup \Or_2]$ is a translation class. It follows that the $4$ homotopic trajectories relative to $\Or_1 \cup \Or_2$:
$$ T_1^+:= \bigcup_{{k\in \Z}} h^k(\alpha_1^+)_\# $$
$$ T_1^-:= \bigcup_{{k\in \Z}} h^k(\alpha_1^-)_\# $$
$$ T_2^+:= \bigcup_{{k\in \Z}} h^k(\alpha_2^+)_\# $$
$$ T_2^-:= \bigcup_{{k\in \Z}} h^k(\alpha_2^-)_\# $$
are proper homotopic lines. Moreover, the $T_i^+$'s (respectively the $T_i^-$'s) are mutually disjoint. Let $\phi$ be a homeomorphism of the plane that preserves the orientation and sends, for $i=1,2$:
\begin{itemize}
\item $T_i^-$ on $\R \times \{i\}$ and $T_2^-$ on $\R \times \{2\}$;
\item $\{x_i\}$ on $(0,i)$ and $\{h(x_i)\}$ on  $(1,i)$, where $x_i$ and $h(x_i)$ are the endpoints of $\alpha_i^-$;
\item $\Or_i$ on $\Z \times \{i\}$.
\end{itemize}

Let $\tau$ be the horizontal translation of the plane which maps $(x,y) \in \R^2$ to $(x+1,y)$.

Let $\pi$ be the quotient map which quotients $\R^2 - (\Z \times \{1,2\})$ by $\tau$ and let $\C_2$ denotes the quotient pointed cylinder.
Note that if we consider $\C_2$ as a vertical cylinder, $\pi(\phi(\alpha_i^-))$ is homotopic in $\C_2$ to a horizontal circle for $i=1,2$ (see figure \ref{figu:definition-tangle} for an example).

\begin{figure}[h!] 
\labellist
\small\hair 2pt
\pinlabel $T(\alpha_i^\pm,h,\Or)$ at 30 480
\pinlabel $T(\phi(\alpha_i^\pm),\tau,\Z \times \{1,2\})$ at 60 270
\pinlabel $\phi$ at 215 305
\pinlabel $\pi$ at 215 150
\endlabellist
\centering
\vspace{0.4cm}
\includegraphics[scale=0.8]{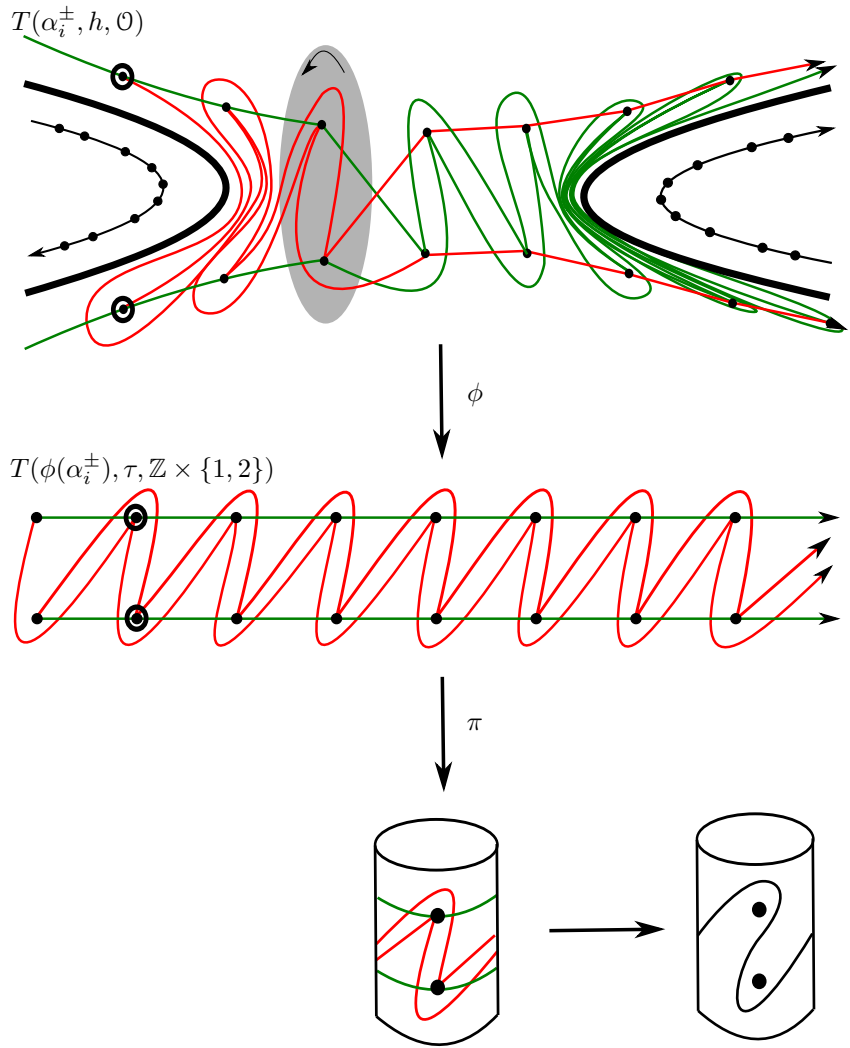}
\caption{Definition of the relative tangle for the example $B$ of section $1$ (Brouwer class of the product of a free half-twist with a flow).}
\label{figu:definition-tangle}
\end{figure}

\begin{lemma}
With the previous notations, the homotopy classes of the arcs $\pi(\phi(\alpha_i^+))$ in $\C_2$ are independent of $\phi$.
\end{lemma}

\begin{proof}
If $\psi$ is another homeomorphism with the same properties, then $\phi$ and $\psi$ coincide on the two topological lines $T_1^-$ and $T_2^-$. According to lemma \ref{lemma:generalized Alexander trick}, $\phi$ and $\psi$ are isotopic relatively to $\Or_1 \cup \Or_2$, hence $\phi(\alpha_i^+)$ is isotopic to $\psi(\alpha_i^+)$ for $i=1,2$.
\end{proof}

Denote by $\gamma$ a curve which is disjoint from $\pi(\phi(\alpha_1^+))$ and $\pi(\phi(\alpha_2^+))$ and which separates $\C_2$ into two cylinders with puncture, each of them containing one of the $\pi(\phi(\alpha_i^+))$'s. 
Note that $\gamma$ is unique up to isotopy in $\C_2$.

\begin{definition}
We say that the isotopy class of $\gamma \in \C_2$ is the \emph{tangle of the irreducible area of $[h;\Or]$ relative to $\al$}.
\end{definition}

\begin{rem}
Note that $\gamma$ is never a horizontal circle. Indeed, we could get a horizontal curve only if we had $\alpha_i^-$ isotopic to $\alpha_i^+$ relatively to $\Or_1 \cup \Or_2$ for $i=1,2$. In this situation we get proper streamlines for every orbit, hence $[h;\Or]$ is a flow class: this gives a contradiction because we assumed that there exists an irreducible area for $[h;\Or]$.
\end{rem}

This relative tangle depend on the choice of the nice family $\al$, hence it is not a conjugacy invariant. However, we have the following lemma:

\begin{lemma}\label{lemma: unicity of the nice family}
Let $[h;\Or]$ be a Brouwer mapping class relatively to $4$ orbits. Suppose that $[h;\Or]$ is not a flow class. Then if $\al$ and $(\beta_i^\pm)_i$  are two nice families for $[h;\Or]$ disjoint from the walls, then for every $i$ there exists $n_i$ such that $\alpha_i^-$, respectively $\alpha_i^+$, is isotopic to $h^{n_i}(\beta_i^-)$, respectively $h^{n_i}(\beta_i^+)$ relatively to $\Or$.
\end{lemma}

\begin{proof}
This will follow from the description of the adjacency areas of $[h;\Or]$. Choose a complete family of adjacency areas for $[h;\Or]$ and a representative $\{ \Delta_1,\Delta_2 \}$ of the set of walls. For every nice family $\al$ disjoint from $\Delta_1 \cup \Delta_2$, there exists $(m_i,n_i)_i \in (\Z^2)^4$ such that $h^m_i(\alpha_i^-)$ and $h^n_i(\alpha_i^+)$ are included in adjacency areas. If we fix the endpoints in the adjacency area, there is only one isotopy class of homotopic translation arc included in the chosen adjacency area and disjoint from $\Delta_1$ and $\Delta_2$: indeed there is only one isotopy class of translation arcs for Brouwer class relatively to one orbit (according to corollary $6.3$ of Handel \cite{Handel}).
\end{proof}

We denote by $T$ the left Dehn twist around a separating horizontal circle between the two punctures in $\C_2$.

\begin{lemma} With the previous notations, if $\gamma \in \C_2$ (respectively $\gamma' \in \C_2$) is the tangle of the irreducible area of $[h;\Or]$ relative to $\al$ (respectively $\bet$), then there exists $n\in \Z$ such that $\gamma = T^n \gamma'$.
\end{lemma}

\begin{proof}
This is a consequence of lemma \ref{lemma: unicity of the nice family}: this lemma implies that if $\phi$ is as in the previous notations, for $i=1,2$ we have: $$T(\phi(\alpha_i^\pm),\tau,\Z\times \{1,2\})=T(\phi(h^{n_i}(\beta_i^\pm)),\tau,\Z\times \{1,2\})=T(\phi(\beta_i^\pm),\tau,\Z\times \{1,2\}).$$
Moreover, we have $\phi(h^{n_i}(x_i))=\tau^{n_i}(\phi(x_i))$. \\
Since $\phi(x_i)=(0,i)$, it follows that $\phi(h^{n_i}(x_i))=(n_i,i)$, hence: $$\pi(\phi(\beta_i^\pm))=T^{n_1-n_2}(\pi(\phi(\alpha_i^\pm))).$$
\end{proof}

\begin{definition}
With the previous notations, we define the \emph{tangle of $[h;\Or]$} to be the isotopy class $\gamma \in \C_2$ up to composition by $T$.
\end{definition}
By convention, we set that every flow mapping class has trivial tangle.

\begin{corollary} \label{coro:couple of inavariant}
The couple $\texttt{(Diagram with walls, Tangle)}$ is a conjugacy invariant for Brouwer mapping classes relatively to $4$ orbits.
\end{corollary}

We will need the following result in the section \ref{subsection:total conjugacy invariant}.
\begin{lemma}\label{lemma: existence of beta-i with the right tangle}
With the previous notations, let $\gamma$ be the tangle of the irreducible area of $[h;\Or]$ relative to $\al$, and let $\gamma'$ be $T^n(\gamma)$ for some $n\in \Z$, where $T$ is the left Dehn twist as above. Then there exists a nice family $(\beta_i^\pm)_i$ for $[h;\Or]$ such that the tangle of the irreducible area of $[h;\Or]$ relative to $(\beta_i^\pm)_i$ is the isotopy class of $\gamma'$.
\end{lemma}

\begin{proof}
Define $\bet$ as $\beta_1^\pm := h^n(\alpha_1^\pm)$ and $\beta_i^\pm:=\alpha_i^\pm$ for $i=2,3,4$.
\end{proof}

\subsubsection{Realized couples (diagram with walls, tangle)}\label{subsection:realized couples}

In section \ref{subsection:total conjugacy invariant}, we will show that the couple $\texttt{(Diagram with walls, Tangle)}$ is a total conjugacy invariant. He we find which couples are realized by Brouwer mapping classes relatively to four orbits. 

\paragraph{Necessary condition to be realized.}
 Not every couple (Diagram with walls, Tangle) can be realized by a Brouwer mapping class. Indeed, some tangles are associated to non determinant diagram with crossing arrows, and some other tangles are associated to non determinant diagram without crossing arrows. To be more precise, denote by $p$ and $q$ the two marked points of $\C_2$, an suppose that $p$ is above $q$. If $\gamma$ is a curve of $\C_2$ representing the tangle, the marked point of $\C_2$ which is \emph{above} $\gamma$ (i.e. in the connected component of the complementary of $\gamma$ which contains the top of the cylinder $\C_2$) can be $p$ or $q$, depending on the tangles.
 
  Moreover, this point represents the orbit whose forward half streamline is above on the picture, hence whose arrow ends above the other on the diagram. It follows that if this point is $p$, then the diagram is without crossing, and if this point is $q$, the diagram has a crossing. We say that such a tangle \emph{adapted} to the diagram. See figure \ref{figu:examples-tangles} for two examples: 
\begin{itemize}
\item On the tangle of the left, $p$ is above $\gamma$, hence it is the tangle of a diagram without crossing;
\item On the tangle of the right, $q$ is above $\gamma$, hence it is the tangle of a diagram with crossing.
\end{itemize}

\begin{figure}[h!] 
\labellist
\small\hair 2pt
\pinlabel $p$ at 158 77
\pinlabel $q$ at 165 61
\pinlabel $p$ at 423 77
\pinlabel $q$ at 433 60
\endlabellist
\centering
\includegraphics[scale=0.8]{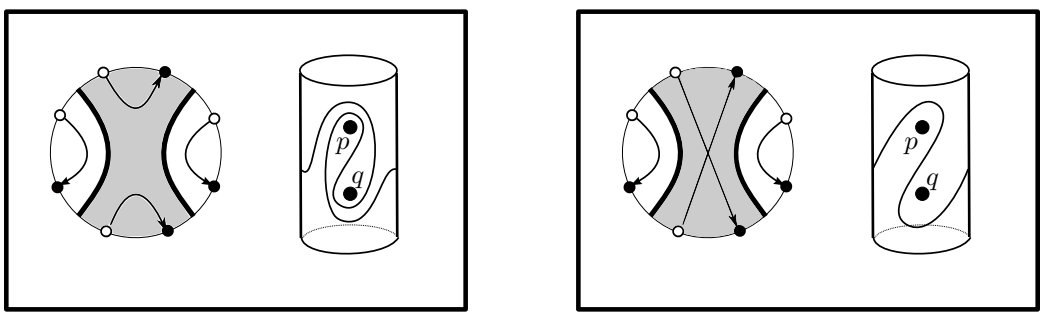}
\caption{Examples of two diagrams with adapted tangles.}
\label{figu:examples-tangles}
\end{figure}

Note that there are infinitely many tangles adapted to each diagram.

\paragraph{Realizing the adapted tangles.} Given a couple (diagram, tangle) such that the tangle is adapted to the diagram, we can produce a Brouwer homeomorphism which realizes this couple as follow. Denote by $(D,\tau)$ the given couple.
\begin{enumerate}
\item \begin{enumerate}
\item If the diagram $D$ does not have crossing arrows, then we define $D'=D$;
\item If the diagram $D$ has crossing arrows, we consider the diagram $D'$ obtained by exchanging the ends of the two crossing arrows. This is a diagram without crossing arrows;
\end{enumerate}
\item We choose a flow $f$ which realizes the diagram $D'$ without walls and such that there exists a $f$-free disk which contains one point of each orbit $\Or_1$ and $\Or_2$ (as in example $A$ of section \ref{section: First tools}), where $\Or_1$ and $\Or_2$ are the two orbits of the irreducible area of $D$;
\item By reversing the process of the definition of the tangle given in section \ref{subsection: tangle}, we get two families of translation arcs. Proposition \ref{prop: deflector} provides us a finite product $\mu$ of mutually disjointly supported $f$-free half twists such that the tangle of $[\mu f;\Or]$ is $\tau$. Note also that the diagram associated to $[\mu f;\Or]$ is $D$.
\end{enumerate}

\subsubsection{Infinitely many Brouwer mapping classes relatively to four orbits}
\begin{prop}
Up to conjugacy, there are (countably) infinitely many Brouwer mapping classes relatively to $4$ orbits.
\end{prop}

\begin{proof}
There are infinitely many tangles adapted to each non determinant diagram (see figure \ref{figu:examples-tangles} for examples). Each of them is realized by a product of a flow with finitely many free half twist disjointly supported (using proposition \ref{prop: deflector}, see the second paragraph of section \ref{subsection:realized couples}). It follows from corollary \ref{coro:couple of inavariant} that there are infinitely many Brouwer mapping classes relatively to $4$ orbits. 
\end{proof}

\subsection{A total conjugacy invariant} \label{subsection:total conjugacy invariant}

In this section we want to show theorem \ref{theo:total invariant}, namely that two Brouwer mapping classes relatively to four orbits are conjugated if and only if they have the same invariant couple:
$$ \texttt{(Diagram with walls, Tangle)}.$$

The following lemma \ref{lemma:total invariant} together with lemma \ref{lemma: existence of beta-i with the right tangle} give the proof of theorem \ref{theo:total invariant}: indeed if $[h;\Or]$ and $[h';\Or']$ have the same invariant couple, then lemma \ref{lemma: existence of beta-i with the right tangle} provides us a nice family $\al$ for $[h;\Or]$ and a nice family $\bet$ for $[h';\Or']$ such that $[h;\Or] $ and $[h';\Or']$ have the same tangle relative to their nice family. Hence they satisfy the hypothesis of lemma \ref{lemma:total invariant}, which says that they are conjugated.

\begin{lemma}\label{lemma:total invariant}
Let $[h;\Or]$ and $[h';\Or']$ be two Brouwer mapping classes relatively to $4$ orbits such that:
\begin{itemize}
\item They have the same diagram with walls;
\item There exist two nice families $\al$ for $[h;\Or]$ and $(\beta_i^\pm)_i$ for $[h';\Or']$ such that the two Brouwer mapping classes have the same relative tangle relative to their nice family;
\end{itemize}
Then $[h;\Or]$ and $[h';\Or']$ are conjugated.
\end{lemma}

\begin{proof}
Note that if the diagram with walls of $[h;\Or]$ has crossing arrows, then this two crossing arrows are in an irreducible area: indeed they are in the same connected component of the walls, which is not a translation area, hence it is an irreducible area (according to theorem \ref{theorem: walls}). If the diagram with walls is without any irreducible area, then $[h;\Or]$ and $[h';\Or']$ are conjugated (this is proposition \ref{coro:determinant diagrams with walls}).

Let us consider the case when there exists an irreducible area. We suppose that the orbits which intersect this area are indexed by $1$ and $2$. Denote by $Z$ and $Z'$ the irreducible areas of $[h;\Or]$ and $[h';\Or']$ respectively. We assume that $h$ preserves $Z$ and $h'$ preserves $Z'$. Denote by $\tau$ the translation of the plane which maps every $(x,y)\in \R^2$ to $(x+1,y)$. Denote by $\phi$ a homeomorphism of the plane as needed to define the tangle for $[h;\Or]$, i.e. which maps:
\begin{itemize}
\item $T_1^- \cup T_2^-$ on $\R \times \{1,2\}$;
\item $\{x_1,x_2\}$ on $\{0\} \times \{1,2\}$, where $x_i$ and $h(x_i)$ are the endpoints of $\alpha_i^-$;
\item $\Or_1 \cup \Or_2$ on $\Z \times \{1,2\}$;
\end{itemize}
where $ T_i^\pm:= \bigcup_{{k\in \Z}} h^k(\alpha_i^\pm)_\# $.

\noindent Likewise, denote by $\psi$ a homeomorphism which maps:
\begin{itemize}
\item $T_1^{'-} \cup T_2^{'-}$ on $\R \times \{1,2\}$;
\item $\{x'_1,x'_2\}$ on $\{0\} \times \{1,2\}$, where $x'_i$ and $h(x'_i)$ are the endpoints of $\beta_i^-$;
\item $\Or'_1 \cup \Or'_2$ on $\Z \times \{1,2\}$;
\end{itemize}
where $ T_i^{'\pm}:= \bigcup_{{k\in \Z}} h^{'k}(\alpha_i^{'\pm})_\# $.

Since the two classes have the same tangle relatively to their nice families, by definition of $\phi$ and $\psi$ we have: $(\phi \alpha_i^\pm)_\#=(\psi \beta_i^\pm)_\#$ for $i=1,2$. Denote by $\gamma_i^\pm$ these arcs.\\

\noindent \textbf{Claim $1$.} There exists $\phi' \in [\phi;\Or_1 \cup \Or_2]$ and $\psi' \in [\psi';\Or'_1 \cup \Or'_2]$ such that $\phi' Z=\psi' Z'$.

\begin{proof}[Proof of the claim] Denote by $\Delta_1$, respectively $\Delta_2$, the boundary component of $Z$ which is disjoint from $T^-_1 \cup T_2^-$, respectively disjoint from $T^+_1 \cup T_2^+$. Since they are disjoint, we may assume that $\phi(\Delta_1)$ is included in the left half plane and $\phi(\Delta_2)$ is included in the right half plane. Similary, we denote by $\Delta'_1$ and $\Delta'_2$ the boundary components of $Z'$, disjoint respectively from $T_1^{'-} \cup T_2^{'-}$ and $T_1^{'+} \cup T_2^{'+}$, and we assume that $\psi(\Delta'_1)$ is included in the left half plane and $\psi(\Delta'_2)$ is included in the right half plane. Now $\phi(\Delta_1)$ and $\psi(\Delta'_1)$ are lines included in the half left strip between $\R \times \{1\}$ and $\R \times \{2\}$: there exists a homeomorphism $\lambda_1$ supported in this half strip and which sends $\phi(\Delta_1)$ on $\psi(\Delta'_1)$. Similary there exists a homeomorphism $\lambda_2$ supported in the right half strip between $\bigcup_{n\geq 0} \tau^n(\gamma_1)$ and $\bigcup_{n\geq 0} \tau^n(\gamma_2)$ and sending $\phi(\Delta_2)$ on $\psi(\Delta'_2)$. It follows that $\lambda_1 \lambda_2 \phi (Z)=\psi(Z')$. Since $\lambda_1 \lambda_2$ is isotopic to the identity relatively to $\R \times \{1,2\}$, $\lambda_1 \lambda_2 \phi$ is isotopic to $\phi$. \end{proof} 

\noindent \textit{Back to the proof of lemma \ref{lemma:total invariant}.} Up to isotopying $\phi$ relatively to $\Or_1 \cup \Or_2$ and $\psi$ relatively to $\Or'_1 \cup \Or'_2$ as in claim $1$, we may assume that $\phi Z=\psi Z'$.

 According to proposition \ref{prop: deflector}, there exists a finite product of $\tau$-free half twists disjointly supported and such that for every sufficiently large $k \in \N$, $(\mu \tau)^k(\gamma_i^-)$ is isotopic relatively to $\Z \times \{1,2\}$ to $\tau^k(\gamma_i^+)$ for $i=1,2$. Since $\mu$ is compactly supported, we can suppose that this support is included in $\phi(Z)=\psi(Z')$.\\

\noindent \textbf{Claim $2$.} With the previous notations, we claim that $[\phi^{-1} \mu \phi h;\Or]$ and $[\psi^{-1} \mu \psi h';\Or']$ are flow classes.\\

\noindent \textit{Proof of the claim.} We do the proof for $[\phi^{-1} \mu \phi h;\Or]$: relatively to $\Or_1 \cup \Or_2$, for every sufficiently large $k \in \N$, for $i=1,2$, $(\phi^{-1} \mu \phi h)^k(\alpha_i^-)$ is isotopic to $h^k(\alpha_i^+)$. Because $\mu$ is supported in $\phi(Z)$, $\phi^{-1} \mu \phi$ is supported in $Z$, and since $h$ preserves $Z$, it follows that $\phi^{-1} \mu \phi h$ also preserves $Z$. Since $\alpha_i^-$ is included in $Z$, and since $\Or_3$ and $\Or_4$ do not intersect $Z$, $(\phi^{-1} \mu \phi h)^k(\alpha_i^-)$ is isotopic to $h^k(\alpha_i^+)$ relatively to $\Or$ (and not only relatively to $\Or_1 \cup \Or_2$). Since $\alpha_i^+$ is a forward proper arc for $[h;\Or]$, it follows that $T(\alpha_i^-,\phi^{-1} \mu \phi h,\Or)$ is a proper streamline. Since $\phi^{-1} \mu \phi$ is supported in $Z$, $h$ is equal to $\phi^{-1} \mu \phi h$ outside $Z$, hence for $j=3,4$, $T(\alpha_j^-,\phi^{-1} \mu \phi h,\Or)=T(\alpha_j^-,h,\Or)$, thus is also a proper streamline. By lemma \ref{lemma:flow classes iff streamlines}, it follows that $[\phi^{-1} \mu \phi h;\Or]$ is a flow class.\qed \\

\noindent \textit{Back to the proof of lemma \ref{lemma:total invariant}.}
Denote by $f$ and $g$ two flows such that:
$$[f;\Or]=[\phi^{-1} \mu \phi h;\Or] \texttt{ and } [g;\Or']=[\psi^{-1} \mu \psi h';\Or'].$$

For every $i$, denote by $T_i$, respectively by $T_i'$, the proper streamline $T(\alpha_i^-,f,\Or)$, respectively  $T(\beta_i^-,g,\Or')$. Changing $\psi$ in the complement of $Z'$ if necessary, we assume that $\phi^{-1} \psi$ maps $T_3'$ on $T_3$, $\Or'_3$ on $\Or_3$, $T_4'$ on $T_4$ and $\Or'_4$ on $\Or_4$. This is possible because the diagrams of $[h;\Or]$ and $[h';\Or']$ are the same.
Note that for every $k \in \Z$, $\phi^{-1} \psi(\psi^{-1} \mu \psi h')^k(\beta)$ is isotopic to $(\phi^{-1} \mu \phi h)^k(\alpha_i^-)$ relatively to $\Or$ for $i=1,2$. Hence $T_i=\phi^{-1} \psi (T_i')$ for every $i$. According to lemma \ref{lemma:generalized Alexander trick}, we get:
$$[\phi^{-1}\psi g \psi^{-1} \phi;\Or]=[f;\Or].$$
Composing both parts by $\phi^{-1} \mu^{-1} \phi$, we can check that:
$$[(\phi^{-1} \psi) h'(\phi^{-1} \psi)^{-1};\Or]=[h;\Or].$$
Hence $[h;\Or]$ and $[h';\Or']$ are conjugated. \end{proof}

\appendix

\section{Diagrams with four orbits}\label{appendix}

Here we represent all the diagrams with four orbits (figures \ref{figu:diag4or_1}, \ref{figu:diag4or_2}, \ref{figu:diag4or_3} and \ref{figu:diag4or_4}). If a diagram can be obtained with a Brouwer mapping class, then we also draw the possible sets of walls, and color in grey the eventual irreducible areas. We put together in the same dashed box the diagrams which are the same without walls but which have different possible sets of walls. We get three different types of diagrams:
\begin{enumerate}
\item The full-framed diagrams are the ones with a Handel's cycle: according to Handel's fixed point theorem (theorem $2.3$ of \cite{Handel}), they cannot be obtained with Brouwer homeomorphisms. Also we forget them to describe Brouwer mapping classes relatively to four orbits;
\item The diagrams without irreducible areas are the determinant ones. Every of them can be realized by a Brouwer mapping class (according to lemma $1.7$ of \cite{FLR2013}). Moreover, up to conjugation, this Brouwer mapping class is unique and it is a flow class (propositions \ref{prop:flow class iff no irreducible area} and \ref{prop:flow implies conjugate});
\item The diagrams with an irreducible area (in grey) are the eight non determinant ones. Up to conjugation, every of them can be realized by infinitely many Brouwer mapping classes. For those diagrams, the tangle allows us to differentiate the different Brouwer mapping classes (see sections \ref{subsection:study of non determinant diagrams} and \ref{subsection:total conjugacy invariant}).
\end{enumerate}
We still denote by $2r'$ the number of families of adjacency.

\begin{figure}[h!] 
\centering
\vspace{0.5cm}
\includegraphics[scale=1.2]{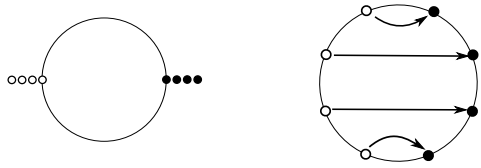}
\caption{Diagram for $r=4$ and $r'=1$.}
\label{figu:diag4or_1}
\end{figure}

\begin{figure}[h!] 
\centering
\vspace{1cm}
\includegraphics[scale=0.9]{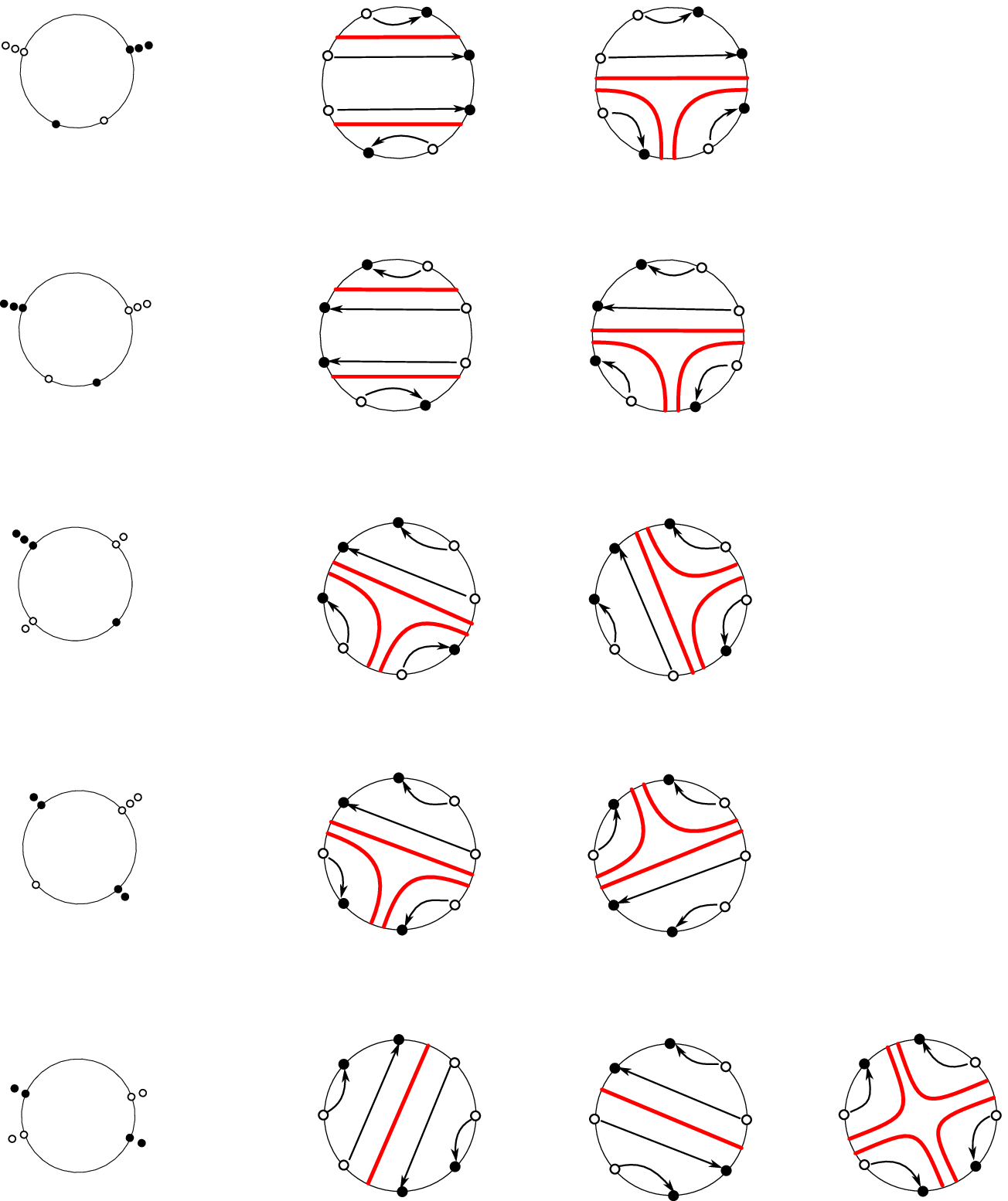}
\caption{Diagram for $r=4$ and $r'=2$.}
\label{figu:diag4or_2}
\end{figure}

\begin{figure}[h!] 
\centering
\includegraphics[scale=0.82]{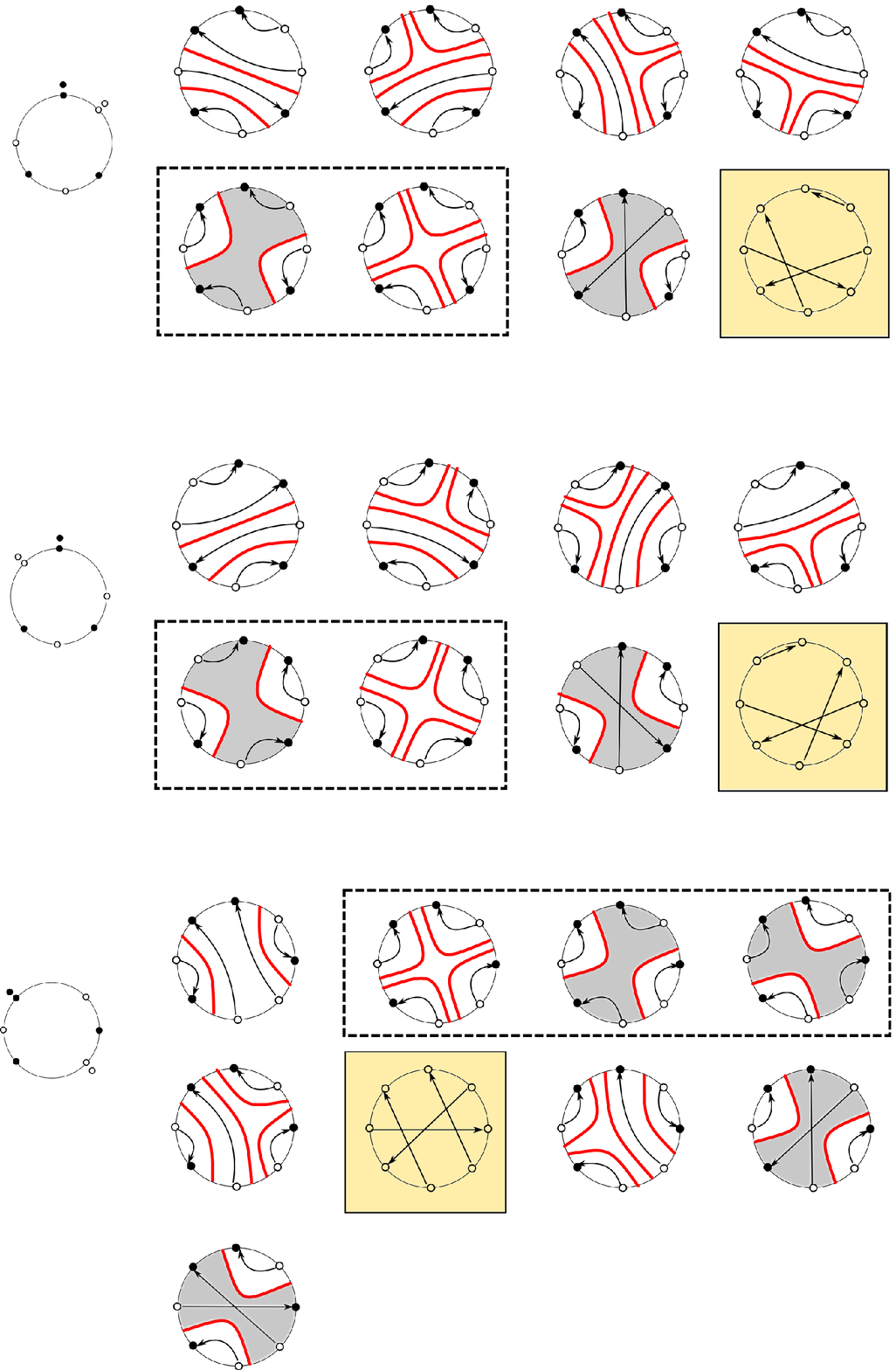}
\caption{Diagram for $r=4$ and $r'=3$.}
\label{figu:diag4or_3}
\end{figure}

\begin{figure}[h!] 
\centering
\includegraphics[scale=0.67]{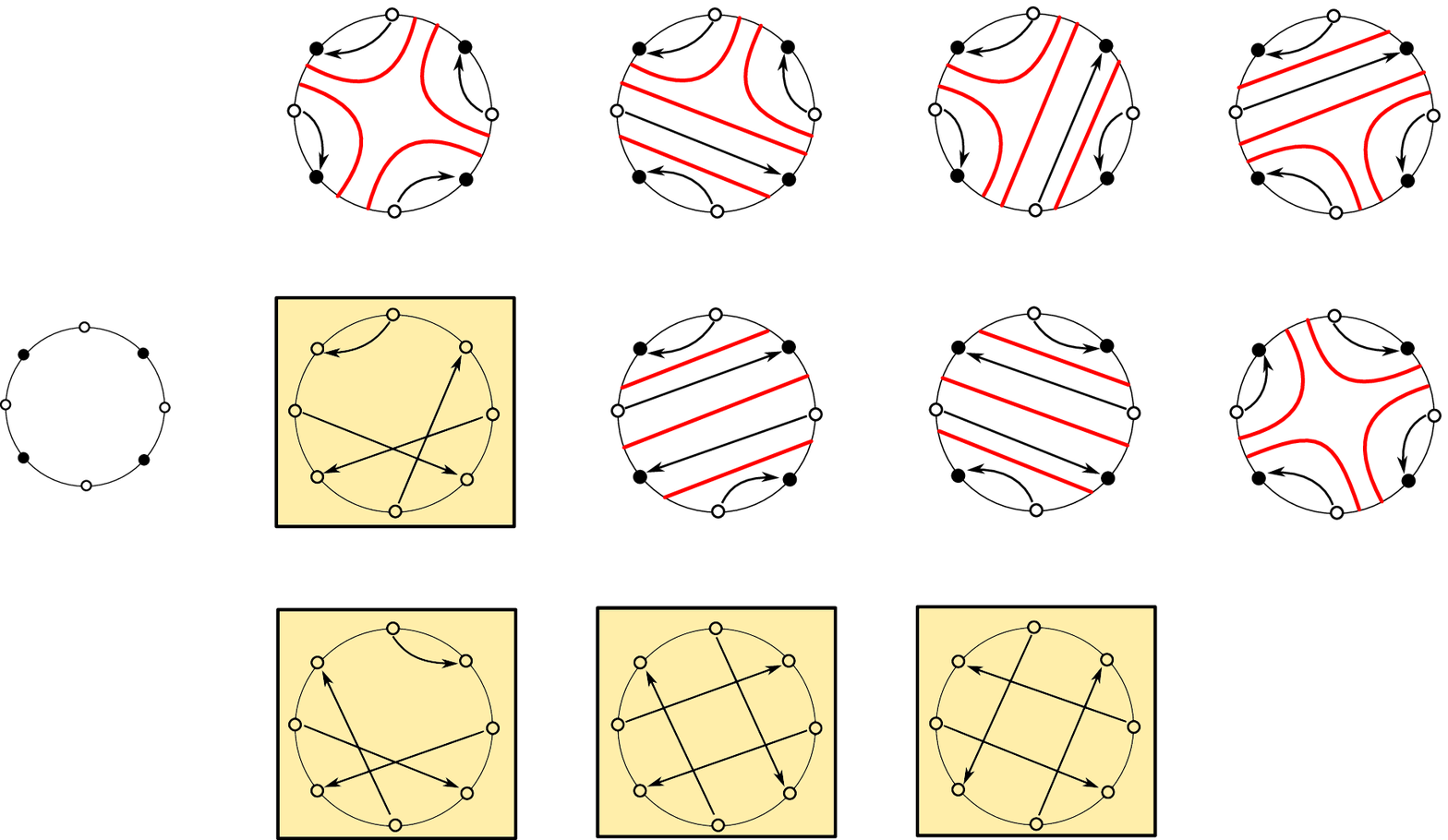}
\caption{Diagram for $r=4$ and $r'=4$.}
\label{figu:diag4or_4}
\end{figure}

\bibliographystyle{alpha}
\bibliography{ref}

\end{document}